\documentclass[11pt]{amsart}

\title{Pinned distances of planar sets with low dimension}

\author{Jacob B. Fiedler}
\address{Department of Mathematics, University of Wisconsin, Madison, Wisconsin 53715}

\email{jbfiedler2@wisc.edu}
\thanks{The first author was supported in part by NSF DMS-2037851 and NSF DMS-2246906.}
\author{D. M. Stull}
\address{Department of Mathematics, University of Chicago, Chicago, IL 60637}
\email{dmstull@uchicago.edu}

\subjclass[2020]{28A78, 28A80, 68Q30}	

\usepackage{amsmath}
\usepackage{amssymb}
\usepackage{mdframed}
\usepackage{amsthm}
\usepackage{mathtools}
\usepackage{enumitem}
\usepackage{placeins}
\usepackage{url}
\usepackage{hyperref}

\newtheorem{thm}{Theorem}
\newtheorem{obs}[thm]{Observation}
\newtheorem{lem}[thm]{Lemma}
\newtheorem{prop}[thm]{Proposition}
\newtheorem{cor}[thm]{Corollary}

\newtheorem*{T1}{Theorem~\ref{thm:maintheorem}}

\newtheorem*{T4}{Theorem~\ref{thm:mainProjection}}

\theoremstyle{remark}

\newtheorem{conj}[thm]{Conjecture}

\DeclareMathOperator{\Dim}{Dim}
\DeclareMathOperator{\dimH}{dim_H}
\DeclareMathOperator{\dimP}{dim_P}

\newcommand{\R}{\mathbb{R}}

\newcommand{\N}{\mathbb{N}}
\newcommand{\Q}{\mathbb{Q}}

\newcommand{\ve}{\varepsilon}
\newcommand{\uhr}{{\upharpoonright}}

\begin{document}
\begin{abstract}
    In this paper, we give improved bounds on the Hausdorff dimension of pinned distance sets of planar sets with dimension strictly less than one. As the planar set becomes more regular (i.e., the Hausdorff and packing dimension become closer), our lower bound on the Hausdorff dimension of the pinned distance set improves. Additionally, we prove the existence of small \textit{universal sets} for pinned distances. In particular, we show that, if a Borel set $X\subseteq\mathbb{R}^2$ is weakly regular ($\dim_H(X) = \dim_P(X)$), and $\dim_H(X) > 1$, then 
    \begin{center}
        $\sup\limits_{x\in X}\dim_H(\Delta_x Y) = \min\{\dim_H(Y), 1\}$,
    \end{center}
    for every Borel set $Y\subseteq\mathbb{R}^2$. Furthermore, if $X$ is also compact and Ahlfors-David regular, then for every Borel set $Y\subseteq\mathbb{R}^2$, there exists some $x\in X$ such that
    \begin{center}
        $\dim_H(\Delta_x Y) = \min\{\dim_H(Y), 1\}$.
    \end{center}
\end{abstract}
\maketitle
\section{Introduction}
Let $E\subseteq\R^n$, and define the distance set of $E$ to be the set
\begin{center}
    $\Delta(E) := \{\vert x - y\vert \mid x,y\in E\}$.
\end{center}
A fundamental problem is to relate the sizes of $E$ and $\Delta(E)$; if $E$ is large, it should contain ``many'' different distances. For finite sets $E$, this is the Erd\H{o}s distinct distances problem, which was essentially solved in the plane by Guth and Katz \cite{GutKat15}. In the continuum setting, Falconer conjectured \cite{Falconer85} that, if $E$ is Borel and has Hausdorff dimension more than $n/2$, then $\Delta(E)$ has positive Lebesgue measure. Falconer's conjecture is still open in all dimensions $n\geq 2$. For $E$ no larger than the $n/2$ threshold, the distance sets do not necessarily have positive measure, but one still expects there to be some lower bound on the Hausdorff dimension of $\Delta(E)$. 

Working in the plane, we improve upon the best known lower bounds of this type when $.5<\dim_H(E)< 1$. In particular, we strengthen bounds on the size of ``pinned'' distance sets. If $E\subseteq \R^n$ and $x\in\R^n$, the pinned distance set of $E$ with respect to $x$ is
\begin{center}
    $\Delta_x(E) := \{\vert x - y\vert\mid y\in E\}$.
\end{center}
Note that if $x\in E$, then $\Delta_x(E)\subseteq \Delta(E)$, so appropriate lower bounds on the size of pinned distance sets imply the same lower bounds for distance sets. The dimension version of Falconer's conjecture for pinned distance sets is the following:
\begin{conj}\label{conj:FalconerPinned}
    Let $E\subseteq\R^n$ be a Borel set with $\dim_H(E) > n/2$. Then 
    \begin{center}
        $\sup\limits_{x\in E} \dim_H(\Delta_x(E)) = 1$.
    \end{center}
\end{conj}
Although a full resolution of Conjecture \ref{conj:FalconerPinned} seems to be beyond current techniques, substantial progress has been made on this question in the past few years \cite{DuZha19,GutIosOuWang20, DuGuOuWanWilZha21,  DuIosOuWanZha21, DuOuRenZhang23a}. Moreover, recent work has proven strong lower bounds on the dimension of pinned distance sets in the plane \cite{Shmerkin17, Shmerkin20,KelShm19,Stull22c}. If $E\subseteq \R^2$ is analytic (i.e. the continuous image of a Borel set), and $d:=\dim_H(E) > 1$, the current records are
$$\sup\limits_{x\in E} \dim_H(\Delta_x (E)) \geq
\begin{cases}
  \min\{\frac{4}{3}d - \frac{2}{3},1\} &\text{Liu \cite{Liu20}}\\
  \min\{\frac{d(4-d)}{5-d},1\} &\text{Fiedler and Stull \cite{FieStu23}}
\end{cases}$$
However, the techniques in both \cite{Liu20} and \cite{FieStu23} rely crucially on the assumption that $\dim_H(E) > 1$. In this paper, we establish improved bounds on the dimension of the pinned distance sets in the absence of this assumption. Let $\dim_P(E)$ denote the packing dimension of $E$ (another notion of fractal dimension that is dual to Hausdorff dimension and always satisfies $\dim_P(E)\geq \dim_H(E)$). Our main theorem is the following:
\begin{thm}\label{thm:maintheorem}
    Suppose $E\subseteq\R^2$ is analytic, $d:= \dim_H(E) \leq 1,$ and $D:= \dim_P(E)$. Then, 
\begin{center}
    $\sup\limits_{x\in E}\dim_H(\Delta_x(E))\geq d \left(1 - \frac{\alpha D - d(D+\alpha -d)}{(d+1)(\alpha D - d^2) -d^2(\alpha + D - 2 d)}\right)$,
\end{center}
where $\alpha = \min\{1 + d, D\}$.
\end{thm}
Note that the bound improves as $E$ becomes more regular, i.e., as the packing dimension of $E$ approaches its Hausdorff dimension. Hence, setting $D=2$ gives the following corollary.
\begin{cor}\label{cor:mainCorollary}
    Suppose $E\subseteq\R^2$ is analytic and $d:= \dim_H(E) \leq 1$. Then, 
\begin{center}
    $\sup\limits_{x\in E}\dim_H(\Delta_x(E))\geq d \left(1 - \frac{2-d}{2(1+2d-d^2)}\right)$.
\end{center}
\end{cor}

Until recently, less progress had been made on the study of pinned distance sets in this setting than above the $n/2$ threshold. A fundamental issue had been the need for improved bounds on the size of radial projections for small sets. Shmerkin and Wang gave the first strong progress on this problem, proving \cite{ShmWang21} that, for any analytic set $E\subseteq \R^2$ with $d:= \dim_H(E)\leq 1$,
\begin{equation}\label{eq:ShmWangBound}
    \sup\limits_{x\in E} \dim_H(\Delta_x E) \geq \frac{d}{2} + \frac{d^2}{2(2+\sqrt{d^2 +4})}.
\end{equation}
Moreover, they proved that if $E$ is weakly regular, i.e., $\dim_H(E) = \dim_P(E)$, then $\sup\limits_{x\in E} \dim_H(\Delta_x E) = d$. A key ingredient of Shmerkin and Wang's proof was a result on radial projections of low-dimensional sets which was then improved in \cite{OrpShmWan22}. In \cite{DuOuRenZhang23a}, Du, Ou, Ren and Zhang proved that
\begin{equation}\label{eq:DuOuRenZhaBound}
    \sup\limits_{x\in E} \dim_H(\Delta_x E) \geq \frac{5}{3}d -1,
\end{equation}
which exceeds (\ref{eq:ShmWangBound}) when $d$ is sufficiently close to one. Corollary \ref{cor:mainCorollary} improves on both of these bounds provided that $d$ is greater than about $.477$. Furthermore, the lower bound is $3/4$ when $d=1$, which agrees with the lower bound in \cite{FieStu23} and eliminates the discontinuity that had existed at this threshold value.

The proof of our main theorem proceeds via ``effective'' methods. In particular, we study the algorithmic information content of points as well as the distances and directions between them. Appropriate bounds on the complexity of distances $\vert x - y\vert$ imply bounds on the size of the pinned distance set $\Delta_x(E)$ via the point-to-set principle of Lutz and Lutz \cite{LutLut17}. The proof broadly follows a similar framework to the authors' previous paper \cite{FieStu23}. However, the assumption that $\dim_H(E) \leq 1$ necessitates new techniques relative to previous work. 

One of the most crucial differences is the use of the recent breakthrough by Orponen, Shmerkin and Wang \cite{OrpShmWan22} on radial projections for low-dimensional sets. In the language of effective dimension, this allows us to find $x,y\in E$ whose direction $e:=\frac{y-x}{\vert x -y\vert}$ is sufficiently complex. Making use of this theorem requires a new reduction, i.e. a proof that the main theorem follows from its effective analog. Additionally, versions of many of the algorithmic tools of \cite{FieStu23} need to be established in this setting, and a careful usage of these tools requires leveraging certain bounds only relevant when the information content of objects is low, as in the case of small sets.

\subsection{Universal sets for pinned distances}

The next main theorems of this paper deal with the concept of \textit{universal sets}. Let $X\subseteq\R^2$, and let $\mathcal{C}$ be a class of subsets of $\R^2$ e.g. Borel sets, analytic sets, or weakly regular sets. We say that $X$ is \textit{weakly universal for pinned distances} for $\mathcal{C}$ if, for every $Y\in \mathcal{C}$ 
\begin{equation}
     \sup\limits_{x\in X}\dim_H(\Delta_x (Y)) = \min\{1, \dim_H(Y)\}.
\end{equation}
Note that the right hand side is as large as the pinned distance set can possibly be in the sense of Hausdorff dimension. We can strengthen this notion by requiring the existence of specific points in $X$ achieving the desired bound. Hence, we say that $X$ is \textit{universal for pinned distances} for $\mathcal{C}$ if, for every $Y\in \mathcal{C}$ there exists some $x\in X$ such that
\begin{equation}
     \dim_H(\Delta_x (Y)) = \min\{\dim_H(Y), 1\}
\end{equation}
It is not obvious that, for a non-trivial class $\mathcal{C}$, small universal sets exist. The final section of this paper establishes the existence of such sets. For weakly universal sets, we have the following:
\begin{thm}\label{thm:mainWeakUniversalSet}
    Let $X\subseteq\R^2$ be a Borel set with $\dim_H(X) = \dim_P(X)$. Then $X$ is weakly universal for pinned distances for the class of Borel sets $Y\subseteq\R^2$ satisfying $\dim_H(Y) < \dim_H(X)$.
\end{thm}
Unpacking the definitions, this theorem shows that there is a \textit{fixed} set $X \subseteq\R^2$ such that, for every $Y$ with $\dim_H(Y) < \dim_H(X)$, $\sup\limits_{x\in X} \dim_H(\Delta_x (Y))$ is maximal. This can be thought of as strengthening the result in \cite{ShmWang21} that if $Y$ is weakly regular, then it has pinned distance sets of the maximum possible size. We are able to obtain the same conclusion using only regularity in the pin set $X$. Moreover, combined with a previous theorem of the authors (Theorem 7 in \cite{FieStu23}, which handles the case that $Y$ is large), we have the following corollary.
\begin{cor}\label{cor:mainCorollaryUniversal}
    Let $X\subseteq\R^2$ be a Borel set with $\dim_H(X) = \dim_P(X) > 1$. Then $X$ is weakly universal for pinned distances for the class of Borel sets $Y\subseteq\R^2$.
\end{cor}
We are able to improve upon this corollary if we further require that $X$ is compact and AD-regular. 
\begin{thm}\label{thm:mainUniversalSet}
    Let $X\subseteq\R^2$ be a compact, Ahlfors-David regular set such that $\dim_H(X)>1$. Then $X$ is universal for pinned distances for the class of Borel sets $Y\subseteq\R^2$.
\end{thm}

For instance, if we take any self-similar four-corner Cantor set $C$ of dimension strictly greater than $1$, then for every Borel set $Y\subseteq\R^2$, there exists some $x\in C$ such that $$\dim_H(\Delta_x Y) = \min\{\dim_H(Y), 1\}.$$

 Considering again the ``unpinned'' distance set problem gives another interpretation of Theorem \ref{thm:mainUniversalSet}. Sets such as $C$ have the property that for any Borel $Y$, one can choose $x\in C$ such that 
 $$\dim_H(\Delta(Y\cup \{x\}))\geq \min\{\dim_H(Y), 1\}$$
 
 \noindent In particular, (the dimension version of) Falconer's conjecture is true in the plane for \emph{any} Borel $Y$, provided that we are allowed to add one point drawn from a small fixed set to $Y$.

\section{Preliminaries}
 
\subsection{Kolmogorov complexity and effective dimension}
We use the algorithmic, or effective, techniques established in \cite{Lutz03a, AthHitLutMay07}. For a more detailed overview of these ideas, and their uses in (classical) geometric measure theory, see, e.g., \cite{LutLutMay22, Lutz17, LutStu18, LutStu20,  Stull22a}.  

The \emph{Kolmogorov complexity} of a point $x\in\R^m$ at \emph{precision} $r\in\N$ is the length of the shortest program $\pi$ that outputs a \emph{precision-$r$} rational estimate for $x$. Formally, this is 
\[K_r(x)=\min\left\{K(p)\,:\,p\in B_{2^{-r}}(x)\cap\Q^m\right\}\,,\]
where $B_{\ve}(x)$ denotes the open ball of radius $\ve$ centered on $x$. Note that this implies that the Kolmogorov complexity of a point is non-decreasing in precision. The \emph{conditional Kolmogorov complexity} of $x$ at precision $r$ given $y\in\R^n$ at precision $s\in\R^n$ is
\[K_{r,s}(x\mid y)=\max\big\{\min\{K_r(p\mid q)\,:\,p\in B_{2^{-r}}(x)\cap\Q^m\}\,:\,q\in B_{2^{-s}}(y)\cap\Q^n\big\}\,.\]
When the precisions $r$ and $s$ are equal, we abbreviate $K_{r,r}(x\mid y)$ by $K_r(x\mid y)$. As a matter of notational convenience, if we are given a non-integral positive real as a precision parameter, we will always round up to the next integer. 

A basic property, proven by Case and J. Lutz~\cite{CasLut15} shows that the growth rate of the Kolmogorov complexity of a point is essentially bounded from above by the dimension of the ambient space. Applying this to $\mathbb{R}^2$ $\mathbb{R}^2$, for any $\ve>0$, for sufficiently large $s$ we have that
\begin{equation*}
K_{r+s}(x)\leq K_r(x)+2 s + \ve s
\end{equation*}

We may \emph{relativize} the definitions in this section to an arbitrary oracle set $A \subseteq \N$. We will frequently consider the complexity of a point $x \in \R^n$ \emph{relative to a point} $y \in \R^m$, i.e., relative to an oracle set $A_y$ that encodes the binary expansion of $y$ is a standard way. We then write $K^y_r(x)$ for $K^{A_y}_r(x)$. Oracle access to the \emph{entire} binary expansion of a point is no less useful than conditional access to that binary expansion only up to a certain precision. Thus, we note that, for every $x\in\R^n$ and $y\in\R^m$,
\begin{equation}\label{eq:OraclesDontIncrease}
K_{s,r}(x\mid y)\geq K^y_s(x) - O(\log r) - O(\log s),
\end{equation}
for every $s, r\in\N$

One of the most useful properties of Kolmogorov complexity is that it obeys the \emph{symmetry of information}. Lemma~\ref{lem:unichain}, an expression of symmetry of information for points in $\mathbb{R}^n$, was proved in \cite{LutStu20}.
\begin{lem}[\cite{LutStu20}]\label{lem:unichain}
	For every $m,n\in\N$, $x\in\R^m$, $y\in\R^n$, and $r,s\in\N$ with $r\geq s$,
	\begin{enumerate}
		\item[\textup{(i)}]$\displaystyle |K_r(x\mid y)+K_r(y)-K_r(x,y)\big|\leq O(\log r)+O(\log\log \vert y\vert)\,.$
		\item[\textup{(ii)}]$\displaystyle |K_{r,s}(x\mid x)+K_s(x)-K_r(x)|\leq O(\log r)+O(\log\log\vert x\vert)\,.$
	\end{enumerate}
\end{lem}
	
A consequence of Lemma \ref{lem:unichain}, is the following. 
\begin{lem}[\cite{LutStu20}]\label{lem:symmetry}
	Let $m,n\in\N$, $x\in\R^m$, $z\in\R^n$, $\ve > 0$ and $r\in\N$. If $K^x_r(z) \geq K_r(z) - O(\log r)$, then the following hold for all $s \leq r$.
	\begin{enumerate}
		\item[\textup{(i)}]$\displaystyle K^x_s(z) \geq K_s(z) - O(\log r)\,.$
		\item[\textup{(ii)}]$\displaystyle K_{s, r}(x \mid z) \geq K_s(x)- O(\log r)\,.$
	\end{enumerate}
\end{lem}

\bigskip

In this paper, we use Mayordomo's characterization \cite{Mayordomo02} of the \emph{effective Hausdorff dimension}  of a point $x\in\R^n$. This is defined as
\[\dim(x)=\liminf_{r\to\infty}\frac{K_r(x)}{r}\,.\]
The \emph{effective packing dimension} of a point $x\in\R^n$ is
\[\Dim(x)=\limsup_{r\to\infty}\frac{K_r(x)}{r}\,.\]
We can relativize both definitions, so that the effective Hausdorff and packing dimension \textit{with respect to an oracle} $A\subseteq \N$ are
\begin{center}
    $\dim^A(x)=\liminf\limits_{r\to\infty}\frac{K^A_r(x)}{r}$ and $\Dim^A(x)=\limsup\limits_{r\to\infty}\frac{K^A_r(x)}{r}$
\end{center}

The following lemma expresses that we can often use an oracle to lower the complexity of points at specified precisions without unduly affecting the complexity of other objects. 

\begin{lem}[\cite{LutStu20}]\label{lem:oracles}
Let $z\in\R^2$, $\eta \in\Q_+$, and $r\in\N$. Then there is an oracle $D=D(r,z,\eta)$ with the following properties.
\begin{itemize}
\item[\textup{(i)}] For every $t\leq r$,
\[K^D_t(z)=\min\{\eta r,K_t(z)\}+O(\log r)\,.\]
\item[\textup{(ii)}] For every $m,t\in\N$ and $y\in\R^m$,
\[K^{D}_{t,r}(y\mid z)=K_{t,r}(y\mid z)+ O(\log r)\,,\]
and
\[K_t^{z,D}(y)=K_t^z(y)+ O(\log r)\,.\]
\item[\textup{(iii)}] If $B\subseteq\N$ satisfies $K^B_r(z) \geq K_r(z) - O(\log r)$, then \[K_r^{B,D}(z)\geq K_r^D(z) - O(\log r)\,.\]
\item[\textup{(iv)}] For every $t\in\N$, $u\in\R^n, w\in\R^m$
\[K_{r,t}(u\mid w) \leq K^D_{r,t}(u\mid w) + K_r(z) - \eta r + O(\log r)\,.\]
\end{itemize}
In particular, this oracle $D$ encodes $\sigma$, the lexicographically first time-minimizing witness to $K(z\uhr r\mid z\uhr s)$, where $s = \max\{t \leq r \, : \, K_{t-1}(z) \leq \eta r\}$.
\end{lem}

\bigskip

\subsection{The Point-to-Set Principle}\label{subsec:ptsp}
The \emph{point-to-set principle} shows that the Hausdorff and packing dimension of a set can be characterized by the effective Hausdorff and effective packing dimension of its individual points. It is the bridge that we use to relate our effective results to our classical theorems in geometric measure theory. 
\begin{thm}[Point-to-set principle~\cite{LutLut18}]\label{thm:p2s}
Let $n \in \N$ and $E \subseteq \R^n$. Then
\begin{equation*}
\dimH(E) = \adjustlimits\min_{A \subseteq \N} \sup_{x \in E} \dim^A(x).
\end{equation*}
\begin{equation*}
\dimP(E) = \adjustlimits\min_{A \subseteq \N} \sup_{x \in E} \Dim^A(x).
\end{equation*}
\end{thm}

The point-to-set principle is extremely useful, but for some applications, we will also need to be able to say something more specific about \emph{which} oracles are Hausdorff oracles for a given set. Hitchcock~\cite{Hitchcock05} and J. Lutz~\cite{Lutz03a} showed that this is possible for a restricted class of sets.

A set $E \subseteq \R^n$ is \textit{effectively compact relative to} $A$ if the set of finite open covers of $E$ by rational balls is computably enumerable relative to $A$. We will use the fact that every compact set is effectively compact relative to some oracle. Hitchcock \cite{Hitchcock05} and J. Lutz \cite{Lutz03a} proved that, if $E$ is effectively compact relative to an oracle $A$, then $A$ is a Hausdorff oracle for $E$ (i.e., $A$ testifies to the point-to-set principle).

\section{Projection theorem}
The purpose of this section is to prove the following (algorithmic) theorem.
\begin{thm}\label{thm:mainProjection}
    Let $x\in \mathbb{R}^2, e\in S^1, \sigma \in \Q \cap (0,1), \ve>0, C\in\mathbb{N}, A\subseteq\mathbb{N}$, and $t<r\in \mathbb{N}$. Let $D \geq \Dim^A(x)$ and $\alpha = \min\{D, 1+\sigma\}$. Suppose that $r$ is sufficiently large and that the following hold
\begin{enumerate}
    \item $0<\sigma<\dim^A(x)\leq 1$,
    \item $t\geq \max\{\frac{\sigma(\sigma + 1 - \alpha)}{\alpha + \sigma + \sigma^2 - \sigma \alpha}r, \frac{r}{C}\}$,
    \item $K_s^{A, x}(e)\geq\sigma s - C\log s$ for all $s\leq t$
\end{enumerate}
Then
\begin{equation*}
    K_r^A(x\vert p_ex, e)\leq K^A_r(x) - \sigma r + \frac{(\alpha - \sigma)\sigma}{\alpha}(r-t) + \ve r.
\end{equation*}

\end{thm}
This theorem gives strong upper bounds on the number of bits needed to compute a point $x\in \R^2$ to precision $r$, if the direction $e$ and the projection $p_e x$ are known to precision $r$. 

The proof of this theorem involves several steps. The first, detailed in Section \ref{ssec:projectionLemmas}, gives tight bounds on the complexity $K^A_{b,b,b,a}(x\mid p_e x, e, x)$ for special intervals of precisions $[a, b]$. The intervals of precisions considered depend on the behavior of the complexity function $K^A(x)$ on the interval $[a,b]$.

The auxiliary results of Section \ref{ssec:projectionLemmas} are only applicable on a given interval of precisions when the complexity function $K_s^A(x)$ is of a certain form on that interval. In Section \ref{ssec:admissiblePartitions} we construct a partition of $[1,r]$ such that every interval in the partition is of the necessary type to apply the results of Section \ref{ssec:projectionLemmas}. 

Finally, in Section \ref{ssec:mainProjTheorem} we use this partition to prove our main theorem.

\subsection{Projection lemmas}\label{ssec:projectionLemmas}
To begin, we recall two lemmas which we will need for the results of this section. 
\begin{lem}[\cite{LutStu20}]\label{lem:intersectionLemmaProjections}
Let $A\subseteq\N$, $x \in \R^2$, $e \in S^{1}$, and $r \in \N$. Let $w \in \R^2$ such that  $p_e x = p_e w$ up to precision $r$. Then 
\begin{equation*}
    K^A_r(w) \geq K^A_t(x) + K^A_{r-t,r}(e\mid x) + O(\log r)\,,
\end{equation*}
where $t := -\log \vert x-w\vert$.
\end{lem}

The next lemma was essentially proved by Lutz and Stull \cite{LutStu20}, with slight modifications for the purposes of this paper.
\begin{lem}\label{lem:enumerationLemmaProjections}
    Suppose that $x\in\mathbb{R}^2$, $e\in S^1$, $t<r\in\mathbb{N}$, $\sigma\in(0, 1]$, $B\subseteq\mathbb{N}$, and $\eta, \ve\in \mathbb{Q}^+$ satisfy the following conditions
\begin{enumerate}
    \item $K^B_r(x)\leq \eta r + \frac{\ve}{2} r$
    \item For every $w\in B_{2^{-t}}(x)$ such that $p_e w = p_e x$,
    \begin{equation*}
        K^B_r(w)\geq \eta r + \min\{\ve r, \sigma(r-s) - \ve r\}
    \end{equation*}
    where $s$ = $-\log\vert x - w\vert \leq r$
\end{enumerate}
Then,
\begin{equation*}
K^{B}_{r, r,r, t}(x\vert p_e x, e, x)\leq \frac{3 \ve}{\sigma}r + K(\ve, \eta) + O(\log r)
\end{equation*}

\end{lem}

Equipped with these facts, we can prove the following two lemmas which bound $K_{r, r, r, t}(x\vert p_e x, e, x)$ on certain kinds of intervals. 
\begin{lem}\label{lem:almostSigmaYellowBound}
    Let $x\in\mathbb{R}^2$, $e\in S^1$, $\ve\in\mathbb{Q}^+$, $C > 0$, $\sigma\in(0, 1]\cap\mathbb{Q}_+$, $A\subseteq{N}$, and $t<r\in\mathbb{N}$ be given. For sufficiently large $r$, if
    \begin{enumerate}
        \item $K_s^{A, x}(e)\geq \sigma s - C \log r$ for all $s \leq r - t$, and
        \item $K^A_{s, t}(x)\geq \sigma(s-t) - \ve r$ for all $t \leq s \leq r$.
    \end{enumerate}
Then, 
    \begin{equation*}
    K^A_{r, r, r, t}(x\vert p_e x, e, x)\leq K^A_{r, t}(x\vert x) - \sigma (r - t) + \frac{7\ve}{\sigma} r
    \end{equation*}
\end{lem}

\begin{proof}
    Let $\eta\in\mathbb{Q}_+$ be such that $\eta r = K^A_t(x) + \sigma(r-t)-3\ve r$. Note that $K(\eta \mid \ve, \sigma) = O(\log r)$. In particular, since $\sigma$ and $\ve$ are fixed rationals and $r$ is sufficiently large, $K(\eta, \ve) = O(\log r)$. Let $D=D^A(r,x, \eta)$ be the oracle of Lemma \ref{lem:oracles}, relative to $A$. Then, $K_r^{A,D}(x)\leq (\eta + \frac{\ve}{2})r$, so the first condition of Lemma \ref{lem:enumerationLemmaProjections} is satisfied relative to $(A,D)$. It remains to verify that the second condition of Lemma \ref{lem:enumerationLemmaProjections} holds. Let $w\in \R^2$ such that $p_e x=p_e w$ and $t\leq -\log \vert x - w\vert =s\leq r$. Then, by Lemma \ref{lem:intersectionLemmaProjections}, 
\begin{align*}
K_r^{A,D}(w) &\geq K_s^{A,D}(x) + K_{r-s, r}^{A,D}(e\vert x) -O(\log r)\\
&\geq K_s^{A,D}(x) + K^A_{r-s, r}(e\vert x) -O(\log r)\\
&\geq K_s^{A,D}(x) + K^{A,x}_{r-s, r}(e) -O(\log r)\\
&\geq K_s^{A,D}(x) + \sigma (r-s) -O(\log r)\\
\end{align*}

By the properties of $D$, there are two cases: either $K_s^{A,D}(x) = K^A_s(x)+O(\log r)$ or $K_s^{A,D}(x) = \eta r+O(\log r)$. In the first case, we have 
\begin{align*}
K_r^{A,D}(w) &\geq K^A_s(x) + \sigma (r-s) -O(\log r)\\
 &\geq K^A_t(x) + K^A_{s, t}(x\vert x) + \sigma (r-s) -O(\log r)\\
  &\geq K^A_t(x) + \sigma(s-t) + \sigma (r-s) -\ve r -O(\log r)\\
 &= K^A_t(x)  + \sigma (r-t) -\ve r -O(\log r)\\
&= (\eta +2\ve)r -O(\log r)\\
&\geq (\eta +\ve)r\\
\end{align*}
In the second case, we have
\begin{align*}
    K_r^{A,D}(w) &\geq \eta r + \sigma (r-s) -O(\log r)\\
    &\geq \eta r + \sigma (r-s) -\ve r\\
\end{align*}
Therefore, the second condition of Lemma \ref{lem:enumerationLemmaProjections} is satisfied, and so we conclude
\begin{equation*}
K^{A,D}_{r, r, r,t}(x\vert p_e x, e,x) \leq \frac{3\ve}{\sigma}r + K(\ve, \eta) +O(\log r)
\end{equation*}
Then by Lemma \ref{lem:oracles}, and the symmetry of information, we have
\begin{align*}
    K^{A}_{r, r, r, t}(x\vert p_ex, e, x) &\leq  K^{A,D}_{r, r, r, t}(x\vert p_ex, e, x) + K_r^A(x) - \eta r + O(\log r)\\
    &\leq \frac{3\ve}{\sigma}r + K(\ve, \eta) + K_r^A(x) - \eta r + O(\log r)\\
    &\leq K^A_{r,t}(x\mid x) - \sigma (r-t) + 3\ve r +\frac{3\ve}{\sigma}r + O(\log r)\\
    &\leq K_{r, t}^A(x\vert x) - \sigma (r-t) + \frac{7\ve}{\sigma}r \\
\end{align*}
\end{proof}

Now, we state the corresponding lemma for the second kind of interval, and prove it (in an almost identical manner).

\begin{lem}\label{lem:almostSigmaTealBound}
    Let $x\in\mathbb{R}^2$, $e\in S^1$, $\ve\in\mathbb{Q}^+$, $C\in\mathbb{N}$, $\sigma\in(0, 1]\cap\mathbb{Q}_+$, $A\subseteq\mathbb{N}$, and $t<r\in\mathbb{N}$ be given. For sufficiently large $r$, if
    \begin{enumerate}
        \item $K_s^{A, x}(e)\geq \sigma s - C \log r$ for every $s \leq r - t$, and
        \item $K^A_{r, s}(x)\leq \sigma(r-s) + \ve r$, for every $t \leq s \leq r$.
    \end{enumerate}
Then, 
    \begin{equation*}
    K^A_{r, r, r, t}(x\vert p_e x, e, x)\leq \frac{7\ve}{\sigma} r
    \end{equation*}
\end{lem}

\begin{proof}
    Pick $\eta\in\mathbb{Q}_+$ be such that $\eta r = K^A_r(x) -3\ve r$. We again note that $K(\eta, \ve) = O(\log r)$. Let $D=D^A(r,x, \eta)$ be the oracle of Lemma \ref{lem:oracles}, relative to $A$. Then, $K_r^{A,D}(x)\leq (\eta + \frac{\ve}{2})r$, so the first condition of Lemma \ref{lem:enumerationLemmaProjections} is satisfied relative to $(A,D)$. It remains to verify the second. Let $w\in \R^2$ such that $p_e x=p_e w$ and $t\leq -\log \vert x - w\vert =s\leq r$. Then, by Lemma \ref{lem:intersectionLemmaProjections}, 
\begin{align*}
K_r^{A,D}(w) &\geq K_s^{A,D}(x) + K_{r-s, r}^{A,D}(e\vert x) -O(\log r)\\
&\geq K_s^{A,D}(x) + \sigma (r-s) -O(\log r)\\
\end{align*}

By the properties of $D$, there are two cases: either $K_s^{A,D}(x) = K^A_s(x)+O(\log r)$ or $K_s^{A,D}(x) = \eta r+O(\log r)$. In the first case, we have 
\begin{align*}
K_r^{A,D}(w) &\geq K^A_s(x) + \sigma (r-s) -O(\log r)\\
 &\geq K^A_s(x) + K^A_{r, s}(x\vert x) -\ve r-O(\log r)\\
 &= K^A_r(x)  -\ve r -O(\log r)\\
&= (\eta +2\ve)r -O(\log r)\\
&\geq (\eta +\ve)r\\
\end{align*}
In the second case, we have
\begin{align*}
    K_r^{A,D}(w) &\geq \eta r + \sigma (r-s) -O(\log r)\\
    &\geq \eta r + \sigma (r-s) -\ve r\\
\end{align*}
Therefore, the second condition of Lemma \ref{lem:enumerationLemmaProjections} is satisfied, and so
\begin{equation*}
K^{A,D}_{r, r, r,t}(x\vert p_e x, e,x) \leq \frac{3\ve}{\sigma}r + K(\ve, \eta) +O(\log r)
\end{equation*}
Then by Lemma \ref{lem:oracles}, and the symmetry of information, we have
\begin{align*}
    K^{A}_{r, r, r, t}(x\vert p_ex, e, x) &\leq  K^{A,D}_{r, r, r, t}(x\vert p_ex, e, x) + K_r^A(x) - \eta r + O(\log r)\\
    &\leq \frac{3\ve}{\sigma}r + K(\ve, \eta) + K_r^A(x) - \eta r + O(\log r)\\
    &\leq 3\ve r +\frac{3\ve}{\sigma}r + O(\log r)\\
    &\leq \frac{7\ve}{\sigma}r.
\end{align*}
\end{proof}

In addition to the previous lemmas, there is another bound which can be useful. In the setting of \cite{FieStu23} and \cite{Stull22c}, since $\sigma=1$, the bound
\begin{equation*}
    K^{A}_{r,r,r,t}(x\mid p_e x, e,x) \leq \max \{K_{r, t}^A(x) - \sigma (r-t), 0\}\leq r-t
\end{equation*}
held automatically on appropriate intervals. We can recover the same bound in the setting $\sigma<1$ with the following lemma.

\begin{lem}\label{lem:oneBoundProjections}
    Let $A\subseteq\N$, $x\in\R^2, e\in\mathcal{S}^1$, and $a \leq b \in \N$. If $b$ is sufficiently large, then
\begin{center}
$K^{A}_{b,b,b,a}(x\mid p_e x, e,x) \leq b-a + O(\log b) $.
\end{center}
\end{lem}
\begin{proof}
The proof follows from three simple observations:
\begin{itemize}
    \item $e^\perp$ is easily computable to precision $b$ given $e$ to precision $b$,
    \item $x= (p_e x) e + (p_{e^\perp} x) e^\perp$ is easily computable to precision $b$ given these four objects to precision $b$. 
    \item $p_{e^\perp} x$ is easily computable to precision $a<b$ given $e^\perp$ to precision $b$ and $x$ to precision $a$
\end{itemize}
Applying these observations in order, have
\begin{align*}
K^{A}_{b,b,b,a}(x\mid p_e x, e,x) &\leq K^{A}_{b,b,b,b,a}(x\mid p_e x, e, e^\perp, x)+ O(\log b)\\
&\leq K^A_{b,b,a}((p_{e^\perp} x) \mid e^\perp, x) + O(\log b)\\
&\leq K^A_{b, a}(p_{e^\perp} x\mid p_{e^\perp} x) + O(\log b).
\end{align*}
Since $p_{e^\perp} x$ is a real number, the conclusion follows.

\end{proof}

\subsection{Admissible partitions}\label{ssec:admissiblePartitions}
For the remainder of this section, we fix a point $x\in\R^2$ and oracle $A\subseteq \N$. First, after \cite{Stull22c}, we extend the complexity function $K_s^A(x)$ to real numbers via $f:\R_+ \rightarrow \R_+$, the piece-wise linear function such that $f(x) = K^A_s(x)$ on the integers such that 
\begin{center}
$f(a) = f(\lfloor a\rfloor) + (a -\lfloor a\rfloor)(f(\lceil a \rceil) - f(\lfloor a\rfloor)) $,
\end{center}
for any non-integer $a$. Now, we can introduce several kinds of intervals corresponding to or identical to those in \cite{Stull22c}.
\begin{itemize}
\item An interval $[a, b]$ is $\boldsymbol{\sigma}$\textbf{-teal} if for every $s\in [a, b]$, $f(b)-f(s)\leq \sigma (b-s)$.
\item An interval $[a, b]$ is $\boldsymbol{\sigma}$\textbf{-yellow} if for every $s\in [a, b]$, $f(s)-f(a)\geq \sigma (s-a)$.
\item An interval $[a, b]$ is $\boldsymbol{\sigma}$\textbf{-green} if it is both $\sigma$-teal and $\sigma$-yellow, and $b-a\leq t$. 
\item An interval $[a, b]$ is \textbf{red} if $f(d) - f(c) > d -c$ for every $a\leq c < d\leq b$.
\item An interval $[a, b]$ is \textbf{blue} if $f$ is constant on $[a, b]$.
\end{itemize}

Intervals of these types have desirable properties. In particular, using Lemmas \ref{lem:almostSigmaYellowBound}, \ref{lem:almostSigmaTealBound} and \ref{lem:oneBoundProjections} we are able to give tight bounds on the complexity $K^A(x\mid p_e x, e, x)$. 
\begin{prop}\label{prop:projectionYellowTeal}
Let $A\subseteq\N$, $x\in\R^2, e\in\mathcal{S}^1, \sigma\in(0,1]\cap\mathbb{Q}_+, \ve>0$, and $a<b\in\R_+$. Suppose that $b$ is sufficiently large and $K^{A,x}_t(e) \geq \sigma t - O(\log b)$, for all $t \leq b-a$. Then the following hold.
\begin{enumerate}
\item If $[a,b]$ is $\sigma$-yellow, 
\begin{center}
$K^{A}_{b,b,b,a}(x\mid p_e x, e,x) \leq \min\{K^A_{b,a}(x\mid x) - \sigma(b-a), b-a\} + \ve b $.
\end{center}
\item If $[a,b]$ is $\sigma$-teal, 
\begin{center}
$K^A_{b,b,b,a}(x\mid p_e x, e,x) \leq \ve b$.
\end{center}
\end{enumerate}
\end{prop}
\begin{proof}
The claim follows from Lemmas \ref{lem:almostSigmaTealBound}, \ref{lem:almostSigmaYellowBound} and \ref{lem:oneBoundProjections}.
\end{proof}

Our goal is to partition the interval $[1, r]$ into teal and yellow intervals, and sum the contributions of each using the symmetry of information. To accomplish all of this, we define an $(M, r, t, \sigma)$\textbf{-admissible} partition to be a partition $[a_i, a_{i+1}]_{i=0}^k$ of $[1, r]$ such that:
    \begin{enumerate}
        \item $k<M$
        \item Every interval in the partition is either $\sigma$-yellow or $\sigma$-teal (or both).
        \item No interval has length more than $t$. 
    \end{enumerate}

\begin{lem}\label{lem:partitionProjectionSum}
Suppose that $A\subseteq\N$, $x \in \R^2$, $e \in \mathcal{S}^1$, $\sigma\in (0,1]\cap\mathbb{Q}_+$, $\ve>0$, $C\in\N$, $t, r \in \N$ satisfy
\begin{enumerate}
    \item $\dim^A(x) > \sigma$,
    \item $K^{A,x}_s(e) \geq \sigma s - C\log s$ for all $s\leq t$.
\end{enumerate}
If $\mathcal{P} = \{[a_i, a_{i+1}]\}_{i=0}^k$ is an  $(10C,r,t, \sigma)$-admissible partition, and $r$ is sufficiently large, then
\begin{align*}
K^A_{r}(x \mid p_e x, e) &\leq \ve r + \sum\limits_{i\in \textbf{Bad}} \min\{K^A_{a_{i+1}, a_{i}}(x \mid x) - \sigma (a_{i+1} - a_i), a_{i+1} - a_i\},
\end{align*}
where
\begin{center}
\textbf{Bad} $=\{i\leq k\mid [a_i, a_{i+1}] \notin T_\sigma\}$.
\end{center}
\end{lem}

\begin{proof}
We assume $r$ is large enough so that, for any $b>\log r$, the conditions of Proposition \ref{prop:projectionYellowTeal} hold for $\frac{\varepsilon}{11 C}$. Starting with $\mathcal{P} $ a $(10C,r,t, \sigma)$-admissible partition, suppose $j$ is the largest integer such that $a_j<\log r$. Remove the intervals $[1, a_1], ...[a_{j-1}, a_j]$ from $\mathcal{P}$ and add $[1, a_j]$, still calling the modified partition $\mathcal{P}$ and re-indexing it for ease of notation. Applying Proposition \ref{prop:projectionYellowTeal}  we see that on the teal intervals of $\mathcal{P}$, 
\begin{equation*}
K^A_{a_{i+1},a_{i+1},a_{i+1},a_i}(x\mid p_e x, e,x) \leq \frac{\ve}{11 C} r,
\end{equation*}
and on the yellow intervals
\begin{equation*}
K^A_{a_{i+1},a_{i+1},a_{i+1},a_i}(x\mid p_e x, e,x) \leq \min\{K^A_{a_{i+1}a_i}(x\mid x) - \sigma (a_{i+1} - a_i), a_{i+1} - a_i\} + \frac{\ve}{11 C} r\\
\end{equation*}

Applying symmetry of information no more than $10 C$ times, we obtain
\begin{align*}
    K^A_{r}(x \mid p_e x, e) &\leq \sum\limits_{ \mathcal{P}} K^A_{a_{i+1}, a_{i+1}, a_{i+1}, a_i}(x \mid p_e x, e, x) + O(\log r)\\
    &\leq \sum\limits_{ \mathcal{P}\setminus [1, a_1]} K^A_{a_{i+1}, a_{i+1}, a_{i+1}, a_i}(x \mid p_e x, e, x) + O(\log r)\\
    &\leq \sum\limits_{i\in\textbf{Teal}} K^A_{a_{i+1}, a_{i+1}, a_{i+1}, a_i}(x \mid p_e x, e, x) \\
    &\qquad +\sum\limits_{i\in\textbf{Bad}} K^A_{a_{i+1}, a_{i+1}, a_{i+1}, a_i}(x \mid p_e x, e, x)+ O(\log r)\\
    &\leq \sum\limits_{i\in\textbf{Bad}} \min\{K^A_{a_{i+1}, a_{i}}(x \mid x) - \sigma (a_{i+1} - a_i), a_{i+1} - a_i\}+\frac{\ve}{11C} \\
    &\qquad  +\sum\limits_{i\in\textbf{Teal}} \frac{\ve}{11 C} r+ O(\log r)\\
    &\leq \sum\limits_{i\in \textbf{Bad}} \min\{K^A_{a_{i+1}, a_{i}}(x \mid x) - \sigma (a_{i+1} - a_i), a_{i+1} - a_i\}+ \frac{10\ve}{11} r + O(\log r).\\
    &\leq \sum\limits_{i\in \textbf{Bad}} \min\{K^A_{a_{i+1}, a_{i}}(x \mid x) - \sigma (a_{i+1} - a_i), a_{i+1} - a_i\}+ \ve r.
\end{align*}
\end{proof}

 An immediate consequence of the above is that, if $B$ is the total length of the bad intervals, 
\begin{equation*}
    K^A_{r}(x \mid p_e x, e) \leq B + \ve r
\end{equation*}

In the proof of Theorem \ref{thm:mainThmEffDim}, we will need a special partition. Specifically, we need a partition which will maximize the proportion of $[1, r]$ covered by of $\sigma$-green intervals. 
\begin{lem}\label{lem:redGreenBluePartitionProjections}
    Let $x\in\mathbb{R}^2, A\subseteq\mathbb{N}$, $\sigma\in(0, 1]\cap \mathbb{Q}$ and $t, r\in\mathbb{N}$ be given. Then there exists a partition $\mathcal{P}$ of $[1, r]$ with the following properties:

    \begin{itemize}
        \item Every interval in $\mathcal{P}$ is red, $\sigma$-green, or blue
        \item $\mathcal{P}$ maximizes the amount of $\sigma$-green, in the sense that, if $s\in[1, r]$ is contained in a $\sigma$-green interval, then $s$ is contained in a $\sigma$-green interval in $\mathcal{P}$.
        \item If $[r_i, r_{i+1}]$ is red, $[r_{i+1}, r_{i+2}]$ cannot be blue. 
        \item The green interval(s) in every red--green--blue sequence of $\mathcal{P}$ have (total) length at least $t$. 
    \end{itemize}

\end{lem}
\begin{proof}
    The construction of such a partition is essentially the same as in \cite{FieStu23}; for completeness, we include the details.
    
    Given $x$, $A$, $t<r$, and $\sigma$, consider the set of all $\sigma$-green intervals. First, we refine this collection via the following greedy strategy.
    \begin{enumerate}
    
    \item Remove any $\sigma$-green interval strictly contained in another from the collection.
    \item Moving from left to right, replace the first overlapping $\sigma$-green intervals $[a, b]$ and $[c, d]$  with $[a, c]$ and $[b, d]$, noting that $[b, d]$ is $\sigma$-green. 
    \item Remove $[b, d]$ if it is now strictly contained in some interval in the remaining collection
    \item Check whether any overlapping $\sigma$-green intervals remain, and if so, repeat steps 2 and 3. Otherwise, label the remaining collection $G$ and stop.
    \end{enumerate}

The elements of $G$ have disjoint interiors, and can be added to $\mathcal{P}$. It is clear that the second property is satisfied.

To finish the proof, observe that a blue interval $[b, c]$ following a red interval $[a, b]$ implies that $b$ is contained in a $\sigma$-green interval of length at most $t$. Thus, there are only three cases for the remaining gaps in $[1, r]\setminus G$: a gap is red, a gap is blue, or a gap is the union of a blue interval and then a red interval. In each case, simply add the correct color of interval(s) to $\mathcal{P}$ to complete the partition. Now, $\mathcal{P}$ consists of only red, $\sigma$-green, and blue intervals, so the first property is satisfied, and the above observation combined with the first two properties gives the third. 

We now observe that the final property is satisfied by $\mathcal{P}$. We may reduce to the case of one $\sigma$-green interval in the sequence: there cannot be two adjacent $\sigma$-green intervals with total length less than $t$. So, it suffices to consider the red--$\sigma$-green--blue sequence $[a, b], [b, c], [c, d]$. Assume $c-b<t$. We know that $\frac{f(c)-f(b)}{c-b}=\sigma$. For all sufficiently small $\delta_1, \delta_2$,  $f(b-\delta_1)< f(b)-\delta_1$ and $f(c+\delta_2)=f(c)$. It follows that there is a $\sigma$-green interval $[a^\prime, c^\prime]$ such that $a^\prime < b$ and $c^\prime > c$, contradicting the construction of $\mathcal{P}$.
\end{proof}

\subsection{Main projection theorem}\label{ssec:mainProjTheorem}
We now use the partition construction of Lemma \ref{lem:redGreenBluePartitionProjections} to prove our main (effective) projection theorem. The proof follows from a careful analysis of this partition.
\begin{T4}
    Let $x\in \mathbb{R}^2, e\in S^1, \sigma \in \Q \cap (0,1), \ve>0, C\in\mathbb{N}, A\subseteq\mathbb{N}$, and $t<r\in \mathbb{N}$. Let $D \geq \Dim^A(x)$ and $\alpha = \min\{D, 1+\sigma\}$. Suppose that $r$ is sufficiently large and that the following hold
\begin{enumerate}
    \item $0<\sigma<\dim^A(x)\leq 1$,
    \item $t\geq \max\{\frac{\sigma(\sigma + 1 - \alpha)}{\alpha + \sigma + \sigma^2 - \sigma \alpha}r, \frac{r}{C}\}$,
    \item $K_s^{A, x}(e)\geq\sigma s - C\log s$ for all $s\leq t$
\end{enumerate}
Then
\begin{equation*}
    K_r^A(x\vert p_ex, e)\leq K^A_r(x) - \sigma r + \frac{(\alpha - \sigma)\sigma}{\alpha}(r-t) + \ve r.
\end{equation*}

\end{T4}

\begin{proof}

First, let $\mathcal{P}$ be the partition of Lemma \ref{lem:redGreenBluePartitionProjections}. Using a greedy strategy, we can partition the red and blue intervals in $\mathcal{P}$ into $\sigma$-yellow and $\sigma$-teal intervals of length at most $t$. Abusing notation, we call this partition $\mathcal{P}$. Let $S$ be the number of $\sigma$-RGB sequences in this partition. We prove the theorem by cases. 

\noindent \textbf{Case $S = 0$}: In this case, $\mathcal{P}$ is almost entirely red and $\sigma$-green. In particular, there is some constant $c$ such that for any real $s>c$, $s$ is not contained in a blue interval of $\mathcal{P}$. A simple greedy construction gives an admissible partition of $[c,r]$ consisting entirely of $\sigma$-yellow intervals (which we will refer to as $\mathcal{P}$). Then, by Lemma \ref{lem:partitionProjectionSum},
\begin{align*}
K^A_{r}(x \mid p_e x, e) &\leq \frac{\ve}{2} r + \sum\limits_{i\in \textbf{Bad}} K^A_{a_{i+1}, a_{i}}(x \mid x) - \sigma (a_{i+1} - a_i)\\
&\leq  (\alpha - \sigma) r+\ve r.
\end{align*}

\noindent \textbf{Case $S = 1$}: Suppose the green portion of our red-green-blue sequence starts at precision $r_1$ and ends at precision $r_1 + s$, for some $s \geq t$. Let $r_2 \geq r_1 + s$ be the minimal precision such that $[r_2, r]$ is the union of yellow intervals; if no such precision exists, then set $r_2 = r$. 

We first assume that $B > \frac{\sigma}{\alpha}(r-t)$, where $B$ is the total length of the non-teal intervals in $\mathcal{P}$. Then, by Lemma \ref{lem:partitionProjectionSum}, we have
\begin{align*}
    K^A_r(x\mid p_e x, e) &\leq K^A_r(x) - \sigma s - \sigma B +\ve r\\
    &< K^A_r(x) - \sigma t - \sigma \frac{\sigma}{\alpha}(r-t) +\ve r\\
    &= K^A_r(x) - \sigma r + \frac{(\alpha- \sigma)\sigma}{\alpha}(r-t) + \ve r,
\end{align*}
as required. Now suppose that $B \leq \frac{\sigma}{\alpha}(r-t)$. Note that this immediately implies that $r_1 \leq \frac{\sigma}{\alpha}(r-t)$.  Using the same argument as the $S=0$ case for the interval $[1, r_1]$ we conclude that
\begin{equation}\label{eq:Sequal1Case1}
    K^A_{r_1}(x \mid p_e x, e) \leq \left(\alpha -\sigma + \frac{\ve}{3}\right) r_1
\end{equation}
Using the same argument for the interval $[r_2, r]$, we have
\begin{equation}\label{eq:Sequal1Case2}
    K^A_{r,r,r,r_2}(x \mid p_e x, e,x) \leq K^A_{r,r_2}(x\mid x) - \sigma(r - r_2) + \frac{\ve}{3}r.
\end{equation}
Finally, since the interval $[r_1, r_2]$ is $\sigma$-green, we have
\begin{equation}\label{eq:Sequal1Case3}
    K^A_{r_2,r_2,r_2,r_1}(x \mid p_e x, e,x) \leq \frac{\ve}{3}r.
\end{equation}
Combining inequalities (\ref{eq:Sequal1Case1}), (\ref{eq:Sequal1Case2}) and (\ref{eq:Sequal1Case3}), the bound on $r_1$, and symmetry of information, yields
\begin{align*}
    K^A_r(x\mid p_e x, e) &\leq \left(\alpha-\sigma\right) r_1 + K^A_{r,r_2}(x\mid x) - \sigma(r - r_2) + \ve r\\
    &\leq \left(\alpha-\sigma\right) r_1 + K^A_r(x) - \sigma r + \ve r\\
    &\leq \left(\alpha - \sigma\right)\frac{\sigma}{\alpha}(r-t) + K^A_r(x) - \sigma r + \ve r,
\end{align*}
and the proof for the $S = 1$ case is complete.

\noindent \textbf{Case $S \geq 2$}: Let $L$ be the total length of the $\sigma$-green intervals, and $B$ be the total length of the non-teal intervals. Recall that, by Lemma \ref{lem:partitionProjectionSum}, $K^A_{r}(x \mid p_e x, e) \leq B$. We also have that
\begin{align*}
    K^A_r(x\mid p_e x, e) &\leq K^A_r(x) - \sigma L - \sigma B + \ve r\\
    &\leq K^A_r(x) - 2\sigma t - \sigma B + \ve r.
\end{align*}
Hence we see that
\begin{equation}\label{eq:Sequal2Case1}
    K^A_r(x\mid p_e x, e) \leq \min\{K^A_r(x) - \sigma B - 2\sigma t, B\} + \ve r.
\end{equation}

By the above inequality, if 
\begin{center}
    $B \leq K^A_r(x) - \sigma r + \frac{(\alpha -\sigma)\sigma}{\alpha}(r-t)$
\end{center}
the proof is complete, so we assume otherwise. In this case, using (\ref{eq:Sequal2Case1}), we see that
\begin{align*}
    K^A_r(x\mid p_e x, e) &\leq K^A_r(x) - 2\sigma t - \sigma B + \ve r\\
    &\leq (1-\sigma)K^A_r(x) + \frac{\sigma^3}{\alpha} r - \frac{2\alpha - (\alpha-\sigma)\sigma}{\alpha}\sigma t+ \ve r.
\end{align*}
Therefore, it suffices to show that
\begin{equation}
    \frac{\sigma (1 +\sigma)}{\alpha} r- \frac{\alpha - \alpha\sigma+ \sigma^2 + \sigma}{\alpha}t \leq K^A_r(x).
\end{equation}
Since $K^A_r(x) \geq \sigma r$, it suffices to show that
\begin{equation}
    \frac{\sigma (1 +\sigma)}{\alpha} r- \frac{\alpha - \alpha\sigma+ \sigma^2 + \sigma}{\alpha}t\leq \sigma r.
\end{equation}
We first note that, if $D \geq 1 + \sigma$, then this holds trivially, and the proof is complete. Otherwise, it suffices to show that
\begin{equation}
    t \geq \frac{\sigma(1+\sigma - \alpha)}{\alpha +\sigma +\sigma^2 - \sigma \alpha}r,
\end{equation}
which is true by our assumed lower bound on $t$.
\end{proof}

\section{Effective dimension of pinned distances}
In this section, we establish results on the complexity of distances. These correspond to the projection results of the previous section. 
\begin{lem}\label{lem:distanceEnumeration}
    Suppose that $B\subseteq\N$, $x, y\in\R^2$, $t<r\in\N$, $\sigma\in(0, 1]$ and $\eta, \ve\in\Q_+$ satisfy the following conditions.
\begin{itemize}
\item[\textup{(i)}]$K^B_r(y)\leq \left(\eta +\frac{\ve}{2}\right)r$.
\item[\textup{(ii)}] For every $z \in B_{2^{-t}}(y)$ such that $\vert x-y\vert = \vert x-z\vert$, \[K^B_{r}(z)\geq \eta r + \min\{\ve r, \sigma (r-s) -\ve r\}\,,\]
where $s=-\log\vert y-z\vert\leq r$.
\end{itemize}
Then,
\[K_{r,t}^{B, x}(y \mid y) \leq K^{B,x}_{r,t}( \vert x-y\vert\mid y) + \frac{3\ve}{\sigma} r + K(\ve,\eta)+O(\log r)\,.\]
\end{lem}
\begin{proof}

The proof is almost identical to that of \cite{Stull22c} (Lemma 12). 

Define an oracle Turing machine $M$ that does the following given oracle $(B,x)$,  and inputs $\pi= \pi_1\pi_2\pi_3\pi_4$ and $q\in\Q^2$ such that $U(\pi_2)=s_1\in\N$, $U(\pi_3) = s_2\in\N$ and $U(\pi_4)=(\zeta, \iota) \in\Q^2$.
				
$M$ first uses $q$ and oracle access to $x$ to compute a dyadic rational $d$ such that $\vert d - \vert x - q\vert \vert < 2^{-s_1}$. $M$ then computes $U^{B,x}(\pi_1, q) = p \in \Q$. For every program $\tau\in\{0,1\}^*$ with $\ell(\tau)\leq \iota s_2 + \frac{\zeta s_2}{2}$, in parallel, $M$ simulates $U(\tau)$. If one of the simulations halts with output $p_2 \in \Q^2 \cap B_{2^{-s_1}}(q)$ such that $\vert d - \vert x-p_2\vert \vert < 2^{-s_2}$, then $M^{A, e}$ halts with output $p_2$. Let $k_M$ be a constant for the description of $M$.
				
Let $\pi_1$, $\pi_2$, $\pi_3$, and $\pi_4$ testify to $K^{B,x}_{r,t}(\vert x-y\vert  \mid y)$, $K(t)$, $K(r)$, and $K(\ve,\eta)$, respectively, and let $\pi= \pi_1\pi_2\pi_3\pi_4$. It can be readily verified that $M^{B, x}(\pi,q)$ eventually halts, and outputs a rational $p \in \Q^2$, whenever $q\in B_{2^{-t}}(y)$.  By the definition of $M$, $p\in B_{2^{-t}}(y)$, and $\vert \vert x-p_2\vert  - \vert x-y\vert  \vert < 2^{1-r}$. Again by the definition of $M$, $K^A(p) \leq (\eta + \frac{\ve}{2})r$. 

Let $w\in\R^2$ such that $\vert w - p\vert < \vert x - y\vert 2^{1-r}$. Since $w\in\R^2$,
\begin{center}
$K^A_r(w) \leq (\eta + \frac{\ve}{2})r + 2\log \lceil \vert x-y\vert \rceil + O(\log r)$.
\end{center}
Let $s = -\log\vert y-w\vert $. We first assume that $\ve r \leq \sigma(r - s) -\ve r$. Then, by condition (ii),
\begin{center}
$\eta r + \ve r \leq (\eta + \frac{\ve}{2})r + 2\log \lceil \vert x-y\vert \rceil + O(\log r)$,
\end{center}
which is a contradiction, since $r$ was chosen to be sufficiently large. Now assume that $\sigma(r-s)-\ve r < \ve r$. Then,
\begin{center}
$\eta r + \sigma(r-s) -\ve r \leq (\eta + \frac{\ve}{2})r + 2\log \lceil \vert x-y\vert \rceil + O(\log r)$,
\end{center}
and so $r - s \leq \frac{3\ve r}{\sigma} + O(\log r)$. Hence,
\begin{center}
$K^{A,x}_r(y\mid w) \leq \frac{3\ve}{\sigma} r + O(\log r)$.
\end{center}

Therefore,
\begin{align*}
K^{A,x}_{r,t}(y\mid y) &\leq \vert \pi \vert + \frac{3\ve}{\sigma} r + O(\log r)\\
&= K^{A,x}_{r,t}(\vert x-y\vert \mid y|) + \frac{3\ve}{\sigma} r + K(\ve, \eta) + O(\log r),
\end{align*}
and the proof is complete.
\end{proof}

For a fixed $y\in \R^2$ and oracle $A\subseteq\N$, we consider the the piece-wise linear function $f$ that agrees with $K^A_s(y)$ on the integers such that 
\begin{center}
$f(a) = f(\lfloor a\rfloor) + (a -\lfloor a\rfloor)(f(\lceil a \rceil) - f(\lfloor a\rfloor)) $,
\end{center}
for any non-integer $a$. We define $\sigma$-yellow and $\sigma$-teal in the exact same way as in Section \ref{ssec:admissiblePartitions}. We say that an interval $[a,b]$ is $\sigma$-green if it is both $\sigma$-yellow and $\sigma$-teal and $b \leq 2a$. 

We now prove the result which bounds the complexity of distances on special (yellow and teal) intervals. This corresponds to our projection result, Proposition \ref{prop:projectionYellowTeal}.
\begin{lem}\label{lem:distancesYellowTeal}
Let $A\subseteq\N$, $x,y \in \R^2$, $e = \frac{x-y}{\vert x-y\vert }$, $\sigma \in \Q\cap (0,1)$, $C\in\N$, $\ve > 0$, and $t < r\leq 2t$. Suppose $r$ is sufficiently large and
\begin{center}
    $K^A_{s,r}(e \mid y) \geq \sigma s - C\log r$, for all $s \leq r - t$.
\end{center}
Then the following hold.
\begin{enumerate}
    \item If $[t,r]$ is $\sigma$-yellow and $t\leq r \leq 2t$,
    \begin{center}
        $K^{A,x}_r(y\mid \vert x - y\vert, y) \leq K^{A}_{r-t}(y\mid y) - \sigma(r-t) + \epsilon r$.
    \end{center}
    \item If $[t,r]$ is $\sigma$-teal and $t\leq r \leq 2t$,
    \begin{center}
        $K^{A,x}_r(y\mid \vert x - y\vert, y) \leq \epsilon r$.
    \end{center}
\end{enumerate}
\end{lem}
\begin{proof}
Let $\ve^\prime \in \Q$ be sufficiently small. Let $\eta \in \Q$ such that 
 
 $$\eta r = \begin{cases}
     K^A_t(y) + \sigma (r - t) - 3\ve^\prime r  &\text{ if $[t,r]$ is yellow}\\
     K^A_r(y) - 3\ve^\prime r &\text{ if $[t,r]$ is teal}
 \end{cases} $$
 Let $D = D^A(r, y, \eta)$ be the oracle of Lemma \ref{lem:oracles}. Note that, by Lemma \ref{lem:oracles}(i), $K^{A,D}_r(y) \leq \eta r + \frac{\ve^\prime r}{2}$.

Suppose that $w\in B_{2^{-t}}(y)$ such that $\vert x - w\vert = \vert x - y\vert $. Let $s = -\log\vert w-y\vert $. Our goal is to show a sufficiently large lower bound on the complexity of $w$, so that we may apply Lemma \ref{lem:distanceEnumeration}. We first note that, since $y$ and $w$ agree to precision at least $t$,
\begin{center}
$\vert p_{e} y - p_{e} w\vert  < 2^{-2t+c}\leq 2^{-r+c}$,
\end{center}
where $c = \log\vert x-y\vert $. Therefore, by Lemma \ref{lem:intersectionLemmaProjections},
\begin{align*}
K^{A,D}_r(w) &\geq K^{A,D}_{r-c}(w) \tag*{}\\
&\geq K^{A,D}_s(y) + K^{A,D}_{r-c-s,r-c}(e\mid y) - O(\log r)\\
&\geq K^{A,D}_s(y) + K^{A,D}_{r-s,r}(e\mid y) - O(\log r)\\
&\geq K^{A,D}_s(y) + K^A_{r-s,r}(e\mid y)  - O(\log r)\\
&\geq K^{A,D}_s(y) + \sigma (r - s) -\frac{\ve^\prime r}{2} - O(\log r).
\end{align*}

By the properties of our oracle $D$, there are two cases to consider. If $K^{A,D}_s(y) = K^A_s(y) - O(\log r)$, then 
\begin{equation}
    K^{A,D}_r(w) \geq K^A_s(y) + \sigma(r - s) -\frac{\ve^\prime r}{2}- O(\log r).
\end{equation}
It is not difficult to verify that, regardless of whether $[t,r]$ is yellow or teal, 
\begin{equation}
    K^{A,D}_r(w) \geq (\eta + \ve^\prime) r.
\end{equation}

For the other case, when $K^{A,D}_s(y) = \eta r - O(\log r)$, we again have
\begin{align*}
K^D_r(w) &\geq  \eta r + \sigma(r - s) -\frac{\ve^\prime r}{2} - O(\log r)\\
&\geq \eta r + \sigma(r- s) -\ve^\prime r
\end{align*}
Thus the conditions of Lemma \ref{lem:distanceEnumeration} are satisfied. Applying this, relative to $(A,D)$, yields
\begin{equation}\label{eq:distanceIntervalIncreasingMain}
K_{r,t}^{A,D,x}(y \mid y) \leq K^{A,D,x}_{r, t}( \vert x-y\vert \mid y) + \frac{3\ve^\prime}{\sigma} r +O(\log r),
\end{equation}
or, equivalently,
\begin{equation}\label{eq:distanceIntervalIncreasingMain2}
K_{r,r,t}^{A,D,x}(y \mid \vert x - y\vert, y) \leq \frac{3\ve^\prime}{\sigma} r +O(\log r).
\end{equation}
Using Lemma \ref{lem:oracles}(iv), we conclude that
\begin{align*}
    K_{r,r,t}^{A,x}(y \mid \vert x - y\vert, y) &\leq  K_{r,r,t}^{A,D,x}(y \mid \vert x - y\vert, y) + K^A_r(y) - \eta r\\
    &\leq K^A_r(y) - \eta r + \frac{3\ve^\prime}{\sigma} r +O(\log r),
\end{align*}
and the conclusion follows.
\end{proof}

As with projections, we have a trivial, though still useful, bound for the complexity of distances.
\begin{lem}\label{lem:oneBoundDistances}
    Let $A\subseteq\N$, $x,y \in \R^2$, $e = \frac{x-y}{\vert x-y\vert }$, $\ve > 0$, and $t < r$. Suppose $r$ is sufficiently large. Then 
    \begin{center}
        $K^{A,x}_{r,r,t}(y\mid \vert x - y\vert, y) \leq r-t + \epsilon r$.
    \end{center}
\end{lem}
\begin{proof}
    Let 
    \begin{center}
        $C= \{z \in B_{2^{-t}}(y)\mid \vert \vert x - z\vert - \vert x - y\vert \vert < 2^{-r}\}$.
    \end{center}
    Note that $C$ can be covered by $O(2^{r-t})$ dyadic cubes with side-length $2^{-r}$. Moreover, this covering is computable given $y$ to precision $t$ and given $x$ as an oracle, and the conclusion follows.
\end{proof}

We say that a partition $\mathcal{P} = \{[a_i, a_{i+1}]\}_{i=0}^k$ of intervals with disjoint interiors is \textbf{\textit{good}} if $[1, r] = \cup_i [a_i, a_{i+1}]$ and it satisfies the following conditions.
\begin{enumerate}
\item $[a_i, a_{i+1}]$ is either $\sigma$-yellow or $\sigma$-teal,
\item $a_{i+1} \leq 2a_i$, for every $i$ and
\item $a_{i+2} > 2a_{i}$ for every $i < k$.
\end{enumerate}

Essentially the same proof as Lemma \ref{lem:partitionProjectionSum} shows that a good partition gives a bound on the complexity of $K^{A,x}_r(y\mid \vert x - y\vert)$.
\begin{lem}\label{lem:boundGoodPartitionDistance}
    Let $\mathcal{P}$ be a good partition. Then
    \begin{center}
        $K^{A,x}_r(y\mid \vert x - y\vert) \leq \epsilon r + \sum\limits_{i\in \mathbf{Bad}} \min\{K^{A}_{r-t}(y\mid y) - \sigma(r-t), r-t\}$,
    \end{center}
    where
    \begin{center}
\textbf{Bad} $=\{i\leq k\mid [a_i, a_{i+1}] \notin T_\sigma\}$.
\end{center}
\end{lem}

\subsection{Constructing a better partition}
We now fix a $A\subseteq\N$, $\sigma \in (0,1)$, $x, y$ and $e:= \frac{y-x}{\vert x -y\vert}$ satisfying
\begin{itemize}
\item[\textup{(C1)}] $\dim^A(x), \dim^A(y) > \sigma$
\item[\textup{(C2)}] $K^{A,x}_r(e) \geq \sigma r - O(\log r)$ for all $r$.
\item[\textup{(C3)}] $K^{A, x}_r(y) \geq K^{A}_r(y) - O(\log r)$ for all sufficiently large $r$. 
\item[\textup{(C4)}] $K^{A}_{t,r}(e\mid y) \geq \sigma t - O(\log r)$ for all sufficiently large $r$ and $t \leq r$.
\end{itemize}

Let $D = \max\{\Dim^A(x), \Dim^A(y)\}$. Let $\alpha = \min\{D, 1+\sigma\}$.

\medskip

\noindent \textbf{Constructing the partition $\mathcal{P}$}: We define the sequence $\{r_i\}$ inductively. To begin, we set $r_0 = r$. Having defined the sequence up to $r_i$, we choose $r_{i+1}$ as follows. Let $a \leq r_i$ be the minimal real such that $[a, r_i]$ can be written as the union of $\sigma$-yellow intervals whose lengths are at most doubling. If $a < r_i$, then we set $r_{i+1} = a$. In this case, we will refer to $[r_{i+1},r_i]$ as a \textbf{yellow} interval of $\mathcal{P}$.

Otherwise,  let $t_i < r_i$ be the maximum of all reals $t<r$ such that
\begin{equation}\label{eq:choiceOfRi+1}
    f(t) = f(r_i) -\sigma (r_i - t) + \frac{(\alpha-\sigma)\sigma}{\alpha}(r_i - 2t).
\end{equation}
Let $t_i^\prime < r_i$ be the largest real $t < r$ such that $f(r_i) = f(t) + \sigma(r_i - t)$. Note that such a $t^\prime$ must exist. We then set 
\begin{equation}
    r_{i+1} = \max\{t_i, t_i^\prime\}.
\end{equation}
Note that, in this case, $[r_{i+1}, r_i]$ is $\sigma$-teal. We therefore refer these intervals as \textbf{teal} intervals of $\mathcal{P}$.

We begin by observing that our partition is well defined. 
\begin{obs}\label{obs:wellDefinedYPartition}
Suppose that $r_i\leq r$ and $r_i > C$, for some fixed constant $C$ depending only on $y$. Then there is at least one $t$ such that
\begin{equation*}
    f(t) = f(r_i) -\sigma (r_i - t) + \frac{(\alpha-\sigma)\sigma}{\alpha}(r_i - 2t).
\end{equation*}
\end{obs}
\begin{proof}
We first note that $f(0) = O(1)$, and so, for $t = 0$,
    \begin{align*}
        f(r_i) -\sigma (r_i-t) + \frac{(\alpha-\sigma)\sigma}{\alpha}(r_i - 2t) &> f(r_i) - \sigma r_i\\
        &> f(0),
\end{align*}
since $\dim^A(y) > \sigma$.

We also see that, at $t = r_i$,
    \begin{equation*}
        f(r_i) -\sigma (r_i-t) + \frac{(\alpha-\sigma)\sigma}{\alpha}(r_i - 2t) < f(r_i),
    \end{equation*}
and the conclusion follows.
\end{proof}

Having established that our partition is well-defined, we now show that the complexity of the distance has good growth on the yellow intervals.
\begin{lem}\label{lem:boundComplexityYellowNewPartition}
Let $\ve > 0$. Suppose that $[r_{i+1}, r_i]\in \mathcal{P}$ is a yellow interval. Then,
\begin{equation*}
    K^{A,x}_{r_i, r_{i+1}}(\vert x -y\vert \mid \vert x-y\vert) \geq \sigma(r_i - r_{i+1}) - \ve r_i.
\end{equation*}
\end{lem}
\begin{proof}
    By assumption, $[r_{i+1}, r_i]$ is the union of of $\sigma$-yellow intervals $[a_{j+1}, a_j]$ such that $a_j \leq 2a_{j+1}$. By a simple greedy strategy we can construct a partition $P_1 = \{[b_{k+1}, b_k]\}$ of $[r_{i+1}, r_i]$ such that, for every $k$, $[b_{k+1}, b_k]$ is yellow, $b_k \leq 2 b_{k+1}$ and $b_{k+2} > 2b_k$. That is, $P_1$ is a good partition of $[r_{i+1}, r_i]$. The conclusion then follows from Lemma \ref{lem:boundGoodPartitionDistance}.
\end{proof}

We now turn to the growth rate of the complexity of $\vert x - y\vert$ on the teal intervals of our partition. This is somewhat involved, and will need several auxiliary results. 
\begin{lem}\label{lem:notTooManyIntervalsYPartition}
If $[r_{i+1}, r_i] \in \mathcal{P}$ is teal, then $r_{i+1} \leq \frac{r_i}{2}$.
\end{lem}
\begin{proof}
Suppose that $[r_{i+1}, r_i] \in \mathcal{P}$ is teal. Then, by the construction of $\mathcal{P}$, $[t, r_i]$ is not $\sigma$-yellow, for any $\frac{r_i}{2}\leq t < r_i$. This immediately implies that $t^\prime_i < \frac{r_i}{2}$. Moreover, for any $t > \frac{r_{i}}{2}$, we see that
\begin{align*}
    f(t) - f(r_i) + \sigma (r_i - t) - \frac{(\alpha-\sigma)\sigma}{\alpha}(r_i - 2t) &\geq \sigma(t-r_i) + \sigma (r_i - t) - \frac{(\alpha-\sigma)\sigma}{\alpha}(r_i - 2t) \\
    &= \frac{(\alpha-\sigma)\sigma}{\alpha}(2t-r_i)\\
    &> 0,
\end{align*}
implying that $t_i \leq \frac{r_i}{2}$, and the conclusion follows.
\end{proof}

We will need a lower bound on the left endpoint of teal intervals.
\begin{lem}\label{lem:leftEndpointNotTooSmall}
    Let $\ve>0$. Suppose that $r_{i+1}$ is sufficiently large and $[r_{i+1}, r_i]\in\mathcal{P}$ is a teal interval. Then 
    \begin{equation*}
       r_{i+1} \geq \frac{(\alpha - \sigma)\sigma}{\alpha D+\alpha \sigma -2\sigma^2}r_i- \ve r_i
    \end{equation*}
\end{lem}
\begin{proof}
Assume that $r_{i+1}$ is large enough that, for all $s>r_{i+1}$, 
\begin{equation*}
d s-\ve^\prime s \leq f(s) \leq D s+\ve^\prime s
\end{equation*}
\noindent for some $\ve^\prime$ to be determined. If $K^A_{r_{i+1}}(y) = K^A_{r_i}(y) - \sigma (r_i - r_{i+1})$, then the conclusion follows by the definition of our partition. Otherwise, we have
\begin{equation}\label{eq:smallerPrecisionBound}
        K^A_{r_{i+1}}(y) \geq K^A_{r_{i}}(y) -\sigma (r_i-r_{i+1}) + \frac{(\alpha-\sigma)\sigma}{\alpha}(r_i - 2r_{i+1}) -\ve^\prime r_i.
    \end{equation}
 Therefore,
    \begin{align*}
        D r_{i+1} &\geq D_y r_{i+1}\\
        &\geq K^A_{r_{i+1}}(y)-\ve^\prime r_{i+1}\\
        &\geq K^A_{r_{i}}(y) -\sigma (r_i-r_{i+1}) + \frac{(\alpha-\sigma)\sigma}{\alpha}(r_i - 2r_{i+1}) -\ve^\prime r_i\\
        &\geq \sigma r_{i+1} + \frac{(\alpha-\sigma)\sigma}{D}(r_i - 2r_{i+1}) -2\ve^\prime r_i\\
        &\geq \frac{(\alpha-\sigma)\sigma}{\alpha}r_i +  \frac{2\sigma^2 -\sigma \alpha}{\alpha}r_{i+1} -2\ve^\prime r_i
    \end{align*}
    Rearranging, choosing $\ve^\prime$ appropriately, and simplifying yields 
    \begin{equation}
        r_{i+1} \geq \frac{(\alpha - \sigma)\sigma}{\alpha D+\alpha \sigma -2\sigma^2}r_i - \ve r_i.
    \end{equation}

\end{proof}

With the previous lemmas, we can prove a strong lower bound on the slope of the complexity of $y$ on teal intervals.
\begin{lem}\label{lem:slopeBoundTeal}
    Let $\ve>0$ be given and suppose that $[r_{i+1}, r_i]\in\mathcal{P}$ is a teal interval and $r_i$ is sufficiently large. Then
    \begin{equation}
        \frac{K^A_{r_i, r_{i+1}}(y\mid y)}{r_i - r_{i+1}} \geq \frac{\sigma^2(D+\alpha-2\sigma)}{\alpha D -\sigma^2}-\ve
    \end{equation}
\end{lem}
\begin{proof}
    Recall that we chose $r_{i+1}$ to be 
    \begin{center}
        $r_{i+1} = \max\{t, t^\prime\}$,
    \end{center}
    where $t^\prime$ is the largest real such that $[t^\prime, r_i]$ is green, and $t$ is the largest real such that
    \begin{equation*}
         f(t) = f(r_i) -\sigma (r_i - t) + \frac{(\alpha-\sigma)\sigma}{\alpha}(r_i - 2t).
    \end{equation*}
    If $r_{i+1} = t^\prime$, then 
    \begin{equation*}
        K^A_{r_{i+1}}(y) = K^A_{r_i}(y) - \sigma(r_i - r_{i+1}),
    \end{equation*}
    and the conclusion holds trivially.

    We now assume that $r_{i+1}=t$. Then we calculate
    \begin{align*}
        K^A_{r_i}(y) - K^A_{r_{i+1}} &= \sigma(r_i - r_{i+1}) - \frac{(\alpha-\sigma)\sigma}{\alpha}(r_i -2r_{i+1})\\
        &= \frac{\sigma^2}{\alpha}r_i - \frac{2\sigma^2 -\alpha\sigma }{\alpha}r_{i+1}.
    \end{align*}
    Hence, it suffices to show that
    \begin{equation}
        \frac{\sigma^2}{\alpha}r_i - \frac{2\sigma^2 -\alpha\sigma}{\alpha}r_{i+1} \geq \left(\frac{\sigma^2(D+\alpha-2\sigma)}{\alpha D -\sigma^2} - \ve\right)(r_i - r_{i+1})
    \end{equation}
    Routine calculation shows that this is true, since
    \begin{equation}
        r_{i+1} \geq \frac{(\alpha - \sigma)\sigma}{\alpha D+\alpha \sigma -2\sigma^2}r_i - \ve r_i.
    \end{equation}
by Lemma \ref{lem:leftEndpointNotTooSmall}.

\end{proof}

We are now able to give strong lower bounds on the complexity $K^{A,x}_{r_i, r_{i+1}}(\vert x - y\vert \mid \vert x - y\vert )$ for any teal interval $[r_{i+1}, r_i] \in \mathcal{P}$. To do so, we will use the following from \cite{Stull22c}.
\begin{lem}\label{lem:lowerBoundOtherPointDistance}
Let $x, y\in \R^2$ and $r\in \N$. Let $z\in \R^2$ such that $\vert x-y\vert = \vert x-z\vert$. Then for every $A\subseteq \N$,
\begin{equation}
K^A_r(z) \geq K^A_t(y) + K^A_{r-t, r}(x\mid y) - K_{r-t}(x\mid p_{e^\prime} x, e^\prime) - O(\log r),
\end{equation}
where $e^\prime = \frac{y-z}{\vert y-z\vert}$ and $t = -\log \vert y-z\vert$.
\end{lem}

\begin{lem}\label{lem:goodtealgrowth}
Let $\ve >0$. Suppose that $[r_{i+1}, r_i] \in \mathcal{P}$ is a teal interval and $r_{i+1}$ is sufficiently large. Then, 
\begin{equation}
    K^{A,x}_{r_i, r_i, r_{i+1}}(y \mid \vert x - y \vert, y) \leq \ve r_i.
\end{equation}
In particular,
$K^{A,x}_{r_i, r_{i+1}}(\vert x - y\vert \mid \vert x - y\vert ) \geq K^{A,x}_{r_i, r_{i+1}}(y\mid y) - \ve r_{i}$.
\end{lem}

\begin{proof}
    Let $\ve\in\mathbb{Q}$ be sufficiently small, assume $r_{i+1}$ is sufficiently large. Let $\eta$ be the rational such that $\eta r_i=K^A_{r_i}(y)-4 \ve r_i$. Let $G = D^A(y, r, \eta)$ be the oracle of Lemma \ref{lem:oracles}. Then, relative to the oracle $(A, G)$, the first condition of Lemma \ref{lem:distanceEnumeration} holds. So, we need to show that for $z\in B_{2^{-r_{i+1}}}(y)$ such that $\vert x - y\vert = \vert x - z\vert$,
    \begin{equation*}
    K^{A, G}_r(z)\geq \eta r_i + \min \{\ve r_i, \sigma(r_i - s) - \ve r_i\}
    \end{equation*}
    where $s=-\log \vert y - z\vert.$

As observed in \cite{Stull22c}, if $s\geq \frac{r_i}{2} - \log r_i$, then $\vert p_e y - p_e z \vert \leq {r_i}^2 2^{-r_i} $. Hence, applying Lemma \ref{lem:intersectionLemmaProjections} and using the properties of our direction $e$, 
\begin{equation*}
    K^{A, G}_{r_i}(z)\geq K^{A, G}_{s}(y) + \sigma(r_i - s) - \frac{\ve}{2} r_i - O(\log r_i). 
\end{equation*}
Using the properties of $G$ and the fact that $[r_{i+1}, r_i]$ is teal establishes the desired bound in this case, so we assume $s\leq \frac{r_i}{2} - \log r_i$. First, apply Lemma \ref{lem:lowerBoundOtherPointDistance} to obtain
\begin{equation*}
    K^{A, G}_r(z) \geq K^{A, G}_t(y) + K^{A, G}_{r-t, r}(x\mid y) - K^{A, G}_{r-t}(x\mid p_{e^\prime} x, e^\prime) - O(\log r)
\end{equation*}
We want to apply Theorem \ref{thm:mainProjection}, so we check that its conditions are satisfied with respect to x, $e^\prime, \ve, C, t=s,$ and $r=r_i-s$. The first condition of the theorem is satisfied using (C1), and the third is satisfied using (C2), so it remains to verify the bound on $t$. We know $r_{i+1}<s$, and can additionally apply Lemma \ref{lem:leftEndpointNotTooSmall} to obtain
\begin{align*}
    s&\geq \dfrac{(\alpha - \sigma)\sigma}{\alpha D + \alpha \sigma -2 \sigma^2} r_i - \ve r_i \\
    &\geq \dfrac{(\alpha - \sigma)\sigma}{\alpha D   - \sigma^2} (r_i-s) - \ve r_i\\
    &\geq \dfrac{\sigma (\sigma + 1-\alpha)}{\alpha + \sigma + \sigma^2 -\alpha \sigma } (r_i - s)
\end{align*}
where the last inequality follows if $\ve$ is small enough depending on $\sigma$ and $D$, which is no significant encumbrance. So, we may apply Theorem \ref{thm:mainProjection}, the properties of $G$, and (C3) with symmetry of information to obtain
\begin{align*}
K^{A, G}_{r_i}(z) &\geq K^{A, G}_s(y) + K^{A, G}_{r_i-s, r_i}(x\mid y) - K^{A, G}_{r_i-s}(x\mid p_{e^\prime} x, e^\prime) - O(\log r)\\
&\geq K^{A}_s(y) + K^{A}_{r_i-s}(x) - K^A_{r_i-s}(x\mid p_{e^\prime} x, e^\prime) - O(\log r)\\
&\geq K^{A}_s(y) + K^{A}_{r_i-s}(x)  - \left(K^A_{r_i-s}(x) - \sigma (r_i-s) + \frac{(\alpha - \sigma)\sigma}{\alpha}(r_i-2s)\right)- \ve r_i-O(\log r)\\
&\geq K^{A}_s(y)  + \sigma (r_i-s) - \frac{(\alpha - \sigma)\sigma}{\alpha}(r_i-2s)- \ve r_i-O(\log r)\\
\end{align*}

We may use \eqref{eq:smallerPrecisionBound} and the teal property to get the bound
\begin{equation*}
        K^A_{s}(y) \geq K^A_{r_{i}}(y) -\sigma (r_i-s) + \frac{(\alpha-\sigma)\sigma}{\alpha}(r_i - 2 s) - \ve r_i
\end{equation*}
which in conjunction with the previous equation gives
\begin{align*}
    K^{A, G}_{r_i}(z)\geq K_{r_i}^A(y) - 3 \ve r_i
    \geq \eta r_i + \ve r_i
\end{align*}
Thus the second condition of Lemma \ref{lem:distanceEnumeration} is satisfied, and applying it gives the desired bound. 
\end{proof}

\subsection{Main effective theorem}
We are now in a position to prove our main effective theorem. This is the effective analog of Theorem \ref{thm:maintheorem}, the main theorem of this paper.
\begin{thm}\label{thm:mainThmEffDim}
Suppose that $x, y\in\R^2$, $e = \frac{y-x}{\vert y-x\vert}$, $0 < \sigma < 1$ and $A\subseteq\N$  satisfy the following.
\begin{itemize}
\item[\textup{(C1)}] $\dim^A(x), \dim^A(y) > \sigma$
\item[\textup{(C2)}] $K^{A,x}_r(e) > \sigma r - O(\log r)$ for all $r$.
\item[\textup{(C3)}] $K^{A, x}_r(y) \geq K^{A}_r(y) - O(\log r)$ for all sufficiently large $r$. 
\item[\textup{(C4)}] $K^{A}_{t,r}(e\mid y) > \sigma t - O(\log r)$ for all sufficiently large $r$ and $t \leq r$.
\end{itemize}
Then 
\begin{equation*}
\dim^{x,A}(\vert x-y\vert) \geq  \sigma \left(1 - \frac{\alpha D - \sigma(D+\alpha -\sigma)}{(\sigma+1)(\alpha D - \sigma^2) -\sigma^2(\alpha + D - 2 \sigma)}\right),
\end{equation*}
where $D = \max\{D_x, D_y\}$ and $\alpha = \min\{D, 1+\sigma\}$. 
\end{thm}
\begin{proof}

Let $\sigma^\prime \in \Q$ satisfying $0 < \sigma^\prime < \sigma$ and let $\ve>0$. Let $r$ be sufficiently large. Let $\mathcal{P} = \{[r_{i+1}, r_i]\}$ be the partition of $[1, r]$ defined in the previous section. Let $L$ be the total length of the $\sigma^\prime$-yellow intervals. Recall that if $[r_{i+1}, r_i]$ is yellow, then we have that
    \begin{equation*}
        K^{A,x}_{r_i, r_{i+1}}(\vert x - y\vert \mid \vert x - y\vert ) \geq \sigma^\prime(r_i - r_{i+1}) - \ve r_{i}
    \end{equation*}

By applications of the symmetry of information and Lemmas
\ref{lem:boundComplexityYellowNewPartition} and \ref{lem:goodtealgrowth} , for sufficiently large $r$ 
\begin{align*}
    K^{A,x}_r(\vert x - y\vert ) &= \sum\limits_{i\in Y} K^{A,x}_{r_i, r_{i+1}}(\vert x - y\vert \mid \vert x - y\vert ) + \sum\limits_{i\in Y^C} K^{A,x}_{r_i, r_{i+1}}(\vert x - y\vert \mid \vert x - y\vert )\\
    &\geq \sigma^\prime L -\frac{\ve}{3} r+ \sum\limits_{i\in Y^C} K^{A,x}_{r_i, r_{i+1}}(\vert x - y\vert \mid \vert x - y\vert )\\
    &\geq \sigma^\prime L -\frac{2\ve}{3}r + \sum\limits_{i\in Y^C} K^{A,x}_{r_i, r_{i+1}}(y\mid y)\\
    &\geq \sigma^\prime L + \frac{\sigma^2(D+\alpha-2\sigma)}{\alpha D -(\sigma^{\prime})^2}(r - L)-\ve r.\label{eq:notmuchteal}
\end{align*}

We also have the bound
\begin{align*}
    d_y r &\leq K^{A,x}_r(y)+\frac{\ve}{3}\\
    &= \sum\limits_{i\in Y} K^{A,x}_{r_i, r_{i+1}}(y\mid y) + \sum\limits_{i\in Y^C} K^{A,x}_{r_i, r_{i+1}}(y\mid y )\\
    &\leq L +\sum\limits_{i\in Y} K^{A,x}_{r_i, r_{i+1}}(\vert x - y\vert \mid \vert x -y\vert )+ \frac{2\ve}{3} \\
    &\;\;\;\;\;\;\;\;\;\;\;+\sum\limits_{i\in Y^C} K^{A,x}_{r_i, r_{i+1}}(\vert x - y\vert \mid \vert x - y\vert)\\
    &\leq L + K^{A,x}_r(\vert x - y\vert ) +\ve r.
\end{align*}
Hence,
\begin{equation}\label{eq:distanceBoundTermsofL}
    K^{A,x}_r(\vert x - y\vert) \geq \max\{\sigma^\prime L + \frac{(\sigma^\prime)^2(D+\alpha-2\sigma^\prime)}{\alpha D -(\sigma^\prime)^2}(r - L), \sigma^\prime r - L\}-\ve r.
\end{equation}

The first term is increasing in $L$, and the second term is decreasing in $L$, so we can set them equal to find the minimum over all $L$, which yields

\begin{equation}
K^{x,A}_r(\vert x-y\vert) \geq \sigma^\prime \left(1 - \frac{\alpha D - \sigma^\prime(D+\alpha -\sigma^\prime)}{(\sigma^\prime+1)(\alpha D - (\sigma^\prime)^2) -(\sigma^\prime)^2(\alpha + D - 2 \sigma^\prime)}\right)r -\ve r.
\end{equation}

\noindent Since we can take $\sigma^\prime$ arbitrarily close to $\sigma$ and $\ve$ as small as desired, the conclusion follows.

\end{proof}

\section{Reduction to the classical bounds}
Recall the conditions on points in our main effective theorem. For an oracle $A\subseteq\N$ we want $x, y\in\R^2$ and $e = \frac{y-x}{\vert x - y\vert}$ to satisfy.
\begin{itemize}
\item[\textup{(C1)}] $\dim^A(x), \dim^A(y) > \sigma$
\item[\textup{(C2)}] $K^{A,x}_r(e) > \sigma r - O(\log r)$ for all $r$.
\item[\textup{(C3)}] $K^{A, x}_r(y) \geq K^{A}_r(y) - O(\log r)$ for all sufficiently large $r$. 
\item[\textup{(C4)}] $K^{A}_{t,r}(e\mid y) > \sigma t - O(\log r)$ for all sufficiently large $r$ and $t \leq r$.
\end{itemize}

We need to prove the existence of points satisfying these conditions.

Let $K, t\geq 0$, $c\in (0,1]$. Let $\mu,\nu \in \mathcal{P}(\R^d)$ with $\text{spt}(\mu) =: X$ and $\text{spt}(\nu) =: Y$. We say that $(\mu, \nu)$ has \textbf{$(t, K, c)$-thin tubes} if there is a Borel set $G \subset X \times Y$ with $(\mu\times\nu)(G) \geq c$ with the following property. If $x\in X$, then
\begin{center}
    $\nu(T\cap G\vert_x) \leq K r^t$ for all $r > 0$ and all $r$-tubes $T$ containing $x$.
\end{center}
We say that $(\mu, \nu)$ has $t$-thin tubes if $(\mu, \nu)$ has $(t, K, c)$-thin tubes for some $K, c > 0$.

We will need the following corollary proved by Orponen, Shmerkin and Wang \cite{OrpShmWan22} (Corollary 2.21).
\begin{cor}\label{cor:OrpShmWan}
Let $s \in (0,1]$ and let $\mu, \nu \in \mathcal{P}(\R^2)$ be measures which satisfy the $s$-dimensional Frostman condition $\mu(B(x,r)) \lesssim r^s$, $\nu(B(x,r)) \lesssim r^s$, for all $x\in \R^2$ and $r > 0$. Assume $\mu(\ell)\nu(\ell)<1$ for every line $\ell$. Then $(\mu, \nu)$ has $\sigma$-thin tubes for all $0\leq \sigma < s$.
\end{cor}

We use the existence of thin tubes to establish the existence of points $x, y$ satisfying properties (C1)-(C3)
\begin{lem}\label{lem:C1-C3}
Let $A\subseteq\N$, $X, Y\subseteq\mathbb{R}^2$ be compact sets, $s \in (0,1]$ and $\mu, \nu \in \mathcal{P}(\R^2)$ be measures whose support is $X$ and $Y$, respectively, which satisfy the $s$-dimensional Frostman condition. Assume $\mu(\ell)\nu(\ell)<1$ for every line $\ell$ and $\nu$ is computable with respect to $A$. Then, for every $\sigma < \min\{s, 1\}$,
\begin{center}
    $\dim_H\{x\in X: \exists y\in Y \text{ satisfying (C1)-(C3)}\}\geq\sigma$
\end{center}
\end{lem}

\begin{proof}
By Corollary \ref{cor:OrpShmWan}, $(\mu, \nu)$ has $\sigma^\prime$-thin tubes, for every $\sigma < \sigma^\prime < s$. Let $G\subseteq X\times Y$ be the Borel set in the definition of $\sigma^\prime$-thin tubes. Since $(\mu \times \nu)(G) \geq c$, the set
\begin{center}
    $X_1 = \{x\in X\mid \nu(G\vert_x) > 0\}$
\end{center}
has positive $\mu$-measure. The Hausdorff dimension of $X_1$ is at least  $s>\sigma$, so, by the point to set principle, the set
\begin{center}
    $X_2 = \{x\in X_1 \mid \dim^A(x) > \sigma\}$
\end{center}
has Hausdorff dimension $s$. 

Let $x$ be any point in $X_2$, and let
$G^\prime \subseteq G\vert_x$ be the set of points in $G\vert_x$ satisfying (C1) and (C3). Then, since $\nu(G\vert_x) > 0$, by \cite{Stull22c} (Lemma 36), $\nu(G^\prime) = \nu(G\vert_x) >0$. 

Since $(\mu, \nu)$ has $\sigma^\prime$-thin tubes, we must have 
\begin{center}
    $\nu(T\cap G^\prime) \leq K r^{\sigma^\prime}$ for all $r > 0$ and all $r$-tubes $T$ containing $x$.
\end{center}
In particular, $\dim_H(\pi_x G^\prime) \geq \sigma^\prime$. By the point-to-set principle, this implies that there is a direction $e \in \pi_x G^\prime$ such that $\dim^{A,x}(e) > \sigma$. Let $e$ be such a direction, and let $y \in G^\prime$ such that $e = \frac{y-x}{\vert x - y\vert}$. Then this choice of $x$ and $y$ satisfies (C1)-(C3).
\end{proof}

Now, we prove that (C1)-(C3) imply (C4)

\begin{lem}\label{lem:C1toC3ImplyC4}
    Assume $A\subseteq\N$, $x, y\in\R^2$, $e = \frac{y-x}{\vert x - y\vert}$ and $\sigma \in (0, 1)$ satisfy (C1)-(C3). Then $x$ and $y$ also satisfy (C4), i.e.,
    \begin{center}
        $K^A_{t, r}(e \mid y) > \sigma t - O(\log r)$,
    \end{center}
    for all sufficiently large $r\in\N$ and $t\leq r$.
\end{lem}
\begin{proof}
    Let $\ve > 0$, $r\in\N$ be sufficiently large and $t\leq r$. By the symmetry of information, and the fact that $e =\frac{y-x}{\vert x - y\vert}$,
    \begin{align*}
        K^{A}_{t, r}(x \mid y) &\leq K^A_{t,r}(e\mid y) + K^A_{t,t,r}(x\mid e, y) + O(\log r)\\
        &\leq K^A_{t,r}(e\mid y) + K^A_{t}(x\mid p_{e^\bot} x, e^\bot) + O(\log r).
    \end{align*}
    We note that $e^\bot$ and $e$ are Lipschitz-bi-computable, and so $K_r(e) = K_r(e^\bot) + O(\log r)$. By a pointwise analog of Marstrand's projection theorem (which was essentially proven in \cite{LutStu18}, but is also a result of the projection bound in section 3 arising from a one interval partition), we have
    \begin{equation*}
     K^A_{t}(x\mid p_{e^\bot} x, e^\bot) < K^A_t(x) - \sigma t + O(\log t),
    \end{equation*}
    and so
    \begin{equation}\label{eq:lemC1toC3ImplyC4_1}
        K^A_{t,r}(e\mid y) \geq \sigma t + K^{A}_{t, r}(x \mid y) - K^A_t(x) - O(\log r).
    \end{equation}
    By condition (C3), and the symmetry of information again,
    \begin{align*}
        K^A_{t,r}(x\mid y) &= K^A_{t}(x) + K^A_{r, t}(y\mid x) - K^A_r(y) - O(\log r)\\
        &\geq K^A_{t}(x) + K^{A,x}_{r}(y) - K^A_r(y) - O(\log r)\\
        &\geq K^A_{t}(x) - O(\log r),
    \end{align*}
    and the conclusion follows.
\end{proof}

We are now in a position to prove our classical bounds. 

\begin{thm}\label{thm:modifiedMainTheorem}
Suppose $X$ and $Y$ are analytic sets such that $\dim_H(X), \dim_H(Y)>\sigma$, and $\dim_P(X), \dim_P(Y)\leq D$. Then for a set of $x\in X$ of Hausdorff dimension at least $\sigma$,
\begin{center}
    $\dim_H(\Delta_x(Y))\geq \sigma \left(1 - \frac{\alpha D - \sigma(D+\alpha -\sigma)}{(\sigma+1)(\alpha D - \sigma^2) -\sigma^2(\alpha + D - 2 \sigma)}\right)$
\end{center}
Where $\alpha = \min\{1 + \sigma, D\}$
\end{thm}
\begin{proof}
Through standard results in geometric measure theory (see, e.g., \cite{Mattila99}), there are compact subsets $X_1\subseteq X$ and $Y_1\subseteq Y$ such that $\dim_H(X_1), \dim_H(Y_1) > \sigma$. To simplify notation, we denote these compact subsets $X$ and $Y$, respectively. In addition, there are measures $\mu, \nu \in \mathcal{P}(\R^2)$ whose support is $X$ and $Y$, and which satisfy the $s$-dimensional Frostman condition, for some $s > \sigma$. 

If, for some line $\ell$, $\mu(\ell)\nu(\ell) = 1$, then the conclusion holds trivially. So we assume that, for every line $\ell$, $\mu(\ell)\nu(\ell) <1$.

Let $A$ the join of the following oracles: an oracle relative to which $Y$ is effectively compact, a packing oracle for $X$, a packing oracle for $Y$, and an oracle relative to which $\nu$ is computable. Applying Lemma \ref{lem:C1-C3} gives the existence of some $x$ and $y$ satisfying conditions (C1)-(C3) relative to $A$. By Lemma \ref{lem:C1toC3ImplyC4}, this pair also satisfies condition (C4). By the properties of $A$, $\Dim^A(x), \Dim^A(y)\leq D$, hence we can apply Theorem \ref{thm:mainThmEffDim} and conclude that
\begin{equation*}
\dim^{A,x}(\vert x-y\vert) \geq \sigma \left(1 - \frac{\alpha D - \sigma(D+\alpha -\sigma)}{(\sigma+1)(\alpha D - \sigma^2) -\sigma^2(\alpha + D - 2 \sigma)}\right)
\end{equation*}

We complete the proof using two observations in \cite{Stull22c}
\begin{itemize}
    \item If $Y$ is effectively compact relative to $A$, $A$ is a Hausdorff oracle for $Y$. 
    \item If $Y$ is effectively compact relative to $A$, then $\Delta_x Y$ is effectively compact relative to $(A, x)$. 
\end{itemize}
By our choice of $A$, then, $(A, x)$ is a Hausdorff oracle for $\Delta_x Y$, and so by the point to set principle 
\begin{equation}\label{eq:lowerBoundHausdorffDimension}
\dim_H(\Delta_x Y) \geq \sigma \left(1 - \frac{\alpha D - \sigma(D+\alpha -\sigma)}{(\sigma+1)(\alpha D - \sigma^2) -\sigma^2(\alpha + D - 2 \sigma)}\right).
\end{equation}
\end{proof}

\begin{cor}
    Suppose $X$ and $Y$ are analytic sets such that $\dim_H(X), \dim_H(Y)>\sigma$. Then for a set of $x\in X$ of Hausdorff dimension at least $\sigma$, 
\begin{center}
    $\dim_H(\Delta_x Y) \geq \sigma \left(1 - \frac{2-\sigma}{2+4\sigma -2\sigma^2}\right)$
\end{center}
\end{cor}
\begin{proof}
    It is easy to verify that the function in the bound (\ref{eq:lowerBoundHausdorffDimension}) is deceasing in $D$ for all $\sigma\in[0, 1]$. Hence, we may assume $D=2$, which implies $\alpha=1+\sigma$, and the conclusion follows.
\end{proof}

\begin{T1}
    Suppose $E\subseteq\R^2$ is analytic set, $d:= \dim_H(E) \leq 1$ and $D:= \dim_P(E)$. Then, 
\begin{center}
    $\sup\limits_{x\in E}\dim_H(\Delta_x(E))\geq d \left(1 - \frac{\alpha D - d(D+\alpha -d)}{(d+1)(\alpha D - d^2) -d^2(\alpha + D - 2 d)}\right)$,
\end{center}
where $\alpha = \min\{1 + d, D\}$. In particular, 
\begin{center}
    $\sup\limits_{x\in E}\dim_H(\Delta_x(E)) \geq d \left(1 - \frac{2-d}{2+4d -2d^2}\right)$
\end{center}
\end{T1}
\begin{proof}
    Since $E$ is analytic, we can find disjoint, analytic subsets $X, Y\subseteq E$ such that $\dim_H(X) = \dim_H(Y) = d$. By Theorem \ref{thm:modifiedMainTheorem}, for every $\sigma < d$, there is a point $x \in E$ such that
    \begin{equation}
\dim_H(\Delta_x Y) \geq \sigma \left(1 - \frac{\alpha D - \sigma(D+\alpha -\sigma)}{(\sigma+1)(\alpha D - \sigma^2) -\sigma^2(\alpha + D - 2 \sigma)}\right),
\end{equation}
    and the conclusion follows.
\end{proof}

\section{Universal pin sets}

Let $\mathcal{C}$ be a class of subsets of $\R^2$. Recall that $X$ is weakly universal for pinned distances for $\mathcal{C}$ if, for every $Y\in \mathcal{C}$
\begin{equation*}
    \sup\limits_{x\in X} \dim_H(\Delta_x (Y)) = \min\{\dim_H(Y), 1\}.
\end{equation*}
and universal for pinned distances for $\mathcal{C}$ if for every $Y\in\mathcal{C}$ there exists some $x\in X$ such that
\begin{equation*}
    \dim_H(\Delta_x (Y)) = \min\{\dim_H(Y), 1\}.
\end{equation*}
In \cite{FieStu23}, the authors established the following result. 
\begin{thm}\label{thm:priorUniversalSets}
Let $X\subseteq\mathbb{R}^2$ be such that $1<d_x<\dim_H(X)=\dim_P(X)$ and let $Y\subseteq \mathbb{R}^2$ be analytic and satisfy $1<d_y<\dim_H(Y)$. Then there is some $F\subseteq X$ such that, 
\begin{center}
    $\dim_H(\Delta_x Y)=1$
\end{center}
\noindent for all $x\in F$. Moreover, $\dim_H(X\setminus F)<\dim_H(X)$. 
\end{thm}

In the language of universal sets, this theorem immediately implies that \emph{every} weakly regular set with Hausdorff dimension greater than one is universal for the class of analytic sets in the plane having dimension more than 1. The regularity of the pin set was essential, as it ensured that for typical pins, there was a nearly optimal projection bound. We will prove two results concerning the case where $Y$ has dimension no more than 1; combined with the above these will give the existence of (weakly) universal sets for classes with no dimension restriction. 

First, we will consider the case that $X$ is weakly regular. In this case, we can guarantee that typical points are arbitrarily close to (effectively) regular, but not necessarily regular (in the sense of having the same effective dimension relative to an appropriate oracle). As Theorem \ref{thm:priorUniversalSets} indicates, if $X$ and $Y$ both have dimension more than 1, this still is enough for the full strength conclusion of universality. However, if $Y$ has dimension at most $1$, we can only achieve weak universality. This is the content of the first subsection. 

In the second subsection, we assume additional regularity in $X$, namely that it is AD regular. Full Ahlfors-David regularity isn't necessary; it is possible to weaken the assumption slightly. However, we will require AD regularity to simplify the exposition. In the AD regular case, we are guaranteed the existence of points that are extremely regular, with their complexity functions only deviating from $\dim_H(x) r$ by a logarithmic term. This enables us to show an optimal effective theorem. We only consider the case that $\dim_H(X)>1$ so we can employ a somewhat different reduction, though one that still depends on \cite{OrpShmWan22}.

\subsection{Weakly regular universal sets}

We begin with a projection bound for the weakly regular case. 

\begin{prop}\label{prop:semiRegularProjection}

Let  $A\subseteq\N$, $x \in \R^2$, $e \in \mathcal{S}^1$, $\sigma \in \Q\cap (0,1)$, $\ve\in \Q^+$, $C, C^\prime\in\N$, and $t, r \in \N$. Suppose that $r$ is sufficiently large, and that the following hold.
\begin{enumerate}
\item[\textup{(P1)}] $\sigma < \dim^A(x)=:d_x$.
\item[\textup{(P2)}] $ t \geq  \frac{r}{C} $.
\item[\textup{(P3)}] $K^{x, A}_s(e) \geq \sigma s - C^\prime\log s$, for all $s \leq t$. 
\end{enumerate}
Then there exists some $\ve_x$ depending only on $d_x$ and $C$ such that $\Dim^A(x)-\dim^A(x)<\ve_x$ implies that 
\begin{equation*}
K^A_r(x \,|\, p_e x, e) \leq K^A_r(x) - \sigma r + \ve r.
\end{equation*}
\end{prop}
\begin{proof}
Let $C$ and $\varepsilon$ be given. Assume $r$ is large enough that if $\sqrt r <s$
\begin{equation*}
  d_x s - \ve^\prime s \leq K_s^A(x) + d_x s - \ve^\prime s   
\end{equation*}
where $\ve^\prime>\ve_x$ and may further depend on $C$, $d_x$, and $\sigma$. Let $t> \frac{r}{C}$ be given and let $\mathcal{P}$ be the partition of $[1, r]$ into red, $\sigma$-green, and blue intervals of Lemma \ref{lem:redGreenBluePartitionProjections}. We will show that after a small precision, all the intervals in this partition are green or red, which will allow us to apply Lemma \ref{lem:partitionProjectionSum} to a partition that almost entirely consists of yellow intervals. 

By the properties of $\mathcal{P}$, and the fact that $d_x > \sigma$, this follows if there are no red $\sigma$-green blue sequences in $\mathcal{P}$ after some small precision. 

Any block of $\sigma$-green intervals in $\mathcal{P}$ always has length at least $t$, and the growth rate is exactly $\sigma$. Suppose $s\in[\sqrt{r}, r]$ and $[s, s+u]$ is a green interval or union of adjacent green intervals. Then 
\begin{equation*}
K_s^A(x) \leq d_x s + \ve^\prime s
\end{equation*}
while
\begin{equation*}
K_{s+u}^A(x) \geq d_x (s+u) - \ve^\prime (s+u)
\end{equation*}
By our assumption, $K$ grows by $\sigma u$ on $[s, s+u]$, hence
\begin{align*}
    \sigma u&\geq (d_x-\ve^\prime)(s+u) - (d_x+\ve^\prime) s\\
    &= d_x u - 2\ve^\prime s - \ve^\prime u\\
    &\geq d_x u - 2\ve^\prime r - \ve^\prime u.
\end{align*}
Using the fact that $t\geq\frac{r}{C}$, this implies
\begin{equation*}
    (d_x - \sigma - \ve^\prime) u \leq 2 \ve^\prime C t
\end{equation*}
So $u<t$ if
\begin{equation*}
\dfrac{2 \ve^\prime C}{d_x - \sigma - \ve^\prime} < 1 . 
\end{equation*}
We are restricted to choosing $\ve^\prime>\ve_x$, but $\ve_x$ can be chosen depending on $\sigma$, $d_x$, and $C$, so we may assume $\ve^\prime$ is small enough that, for sufficiently large $r$, $u<t$ and consequently there are no red-green-blue sequences in $\mathcal{P}$. 

The argument that, for $r$ sufficiently large, $[\sqrt{r}, \frac{\ve}{2} r]$ intersects a red interval proceeds in the same way as the proof of Proposition 40 in \cite{FieStu23}. Hence, we omit it and observe the consequence that, for sufficiently large $r$, $[\frac{ve}{2}r, r]$ is covered entirely by red and green intervals. Using a simple greedy strategy, we can cover $[\frac{ve}{2}r, r]$ by no more than $2C$ yellow intervals. By Lemma \ref{lem:partitionProjectionSum} with respect to $\frac{\ve}{2}$, 

\begin{align*}
    K^A_{r}(x \mid p_e x, e) &\leq \frac{\ve r}{2} + \sum\limits_{i\in \textbf{Bad}} \min\{K^A_{a_{i+1}, a_{i}}(x \mid x) - \sigma (a_{i+1} - a_i), a_{i+1} - a_i\}\\
    &\leq \min\{K_r^A(x) - \sigma r, r\} + \ve r 
\end{align*}
\end{proof}

\begin{thm}\label{thm:mainThmEffDim2}
Suppose that $x, y\in\R^2$, $e = \frac{y-x}{\vert y-x\vert}$, $\sigma \in \Q \cap (0,1)$, $\sigma^\prime > \sigma$, $\ve > 0$, and $A\subseteq\N$  satisfy the following.
\begin{itemize}
\item[\textup{(C1)}] $\dim^A(x)>\sigma^\prime, \dim^A(y) > \sigma$
\item[\textup{(C2)}] $K^{A,x}_r(e) > \sigma r - O(\log r)$ for all $r$.
\item[\textup{(C3)}] $K^{A, x}_r(y) \geq K^{A}_r(y) - O(\log r)$ for all sufficiently large $r$. 
\item[\textup{(C4)}] $K^{A}_{t,r}(e\mid y) > \sigma t - O(\log r)$ for all sufficiently large $r$ and $t \leq r$.
\end{itemize}
Then there is a constant $\delta(\ve, \sigma^\prime) > 0$, depending only on $\ve$ and $s$ such that, if  $\Dim^A(x) - \dim^A(x) < \delta(\ve, \sigma^\prime)$,
\begin{equation*}
\dim^{x,A}(\vert x-y\vert) \geq  \sigma - \frac{10\ve}{\sigma}.
\end{equation*}
\end{thm}
\begin{proof}
 Let $C = \frac{1}{10\ve}$. Let $\delta(\ve, \sigma^\prime)$ be the real $\ve_x$ in the statement of Proposition \ref{prop:semiRegularProjection}, and suppose that $\Dim^A(x) - \dim^A(x) < \delta(\ve, \sigma^\prime)$.
 
Let $\eta$ be a rational such that $\eta \leq \sigma - 3\ve$. Let $r$ be sufficiently large and let $t = \frac{r}{C}$. Let $G = D^A(r,y, \eta)$ be the oracle of Lemma \ref{lem:oracles}. Then, relative to the oracle $(A, G)$, the first condition of Lemma \ref{lem:distanceEnumeration} holds. So, we need to show that for $z\in B_{2^{-t}}(y)$ such that $\vert x - y\vert = \vert x - z\vert$,
    \begin{equation*}
    K^{A, G}_r(z)\geq \eta r + \min \{\ve r, \sigma(r - s) - \ve r\}
    \end{equation*}
    where $s=-\log \vert y - z\vert.$

If $s\geq \frac{r}{2} - \log r$, then $\vert p_e y - p_e z \vert \leq {r}^2 2^{-r} $. Lemma \ref{lem:intersectionLemmaProjections} shows that
\begin{equation*}
    K^{A, G}_{r}(z)\geq K^{A, G}_{s}(y) + \sigma(r - s) - \frac{\ve}{2} r - O(\log r). 
\end{equation*}
Using the properties of $G$ establishes the desired bound in this case, so we assume $s\leq \frac{r}{2} - \log r$. By Lemma \ref{lem:lowerBoundOtherPointDistance} we see that
\begin{equation*}
    K^{A, G}_r(z) \geq K^{A, G}_s(y) + K^{A, G}_{r-s, r}(x\mid y) - K^{A, G}_{r-s}(x\mid p_{e^\prime} x, e^\prime) - O(\log r)
\end{equation*}
It is straightforward to verify that the conditions of Proposition \ref{prop:semiRegularProjection} are satisfied with respect to $x$, $e^\prime, \ve, C, t=s,$ and $r=r-s$. Applying Proposition \ref{prop:semiRegularProjection} yields
\begin{align*}
K^{A, G}_{r}(z) &\geq K^{A, G}_s(y) + K^{A, G}_{r-s, r}(x\mid y) - K^{A, G}_{r-s}(x\mid p_{e^\prime} x, e^\prime) - O(\log r)\\
&\geq K^{A}_s(y) + K^{A}_{r-s}(x) - K^A_{r-s}(x\mid p_{e^\prime} x, e^\prime) - O(\log r)\\
&\geq K^{A}_s(y) + K^{A}_{r-s}(x)  - \left(K^A_{r-s}(x) - \sigma (r-s) + \ve r \right)- \ve r-O(\log r)\\
&\geq K^{A}_s(y)  + \sigma (r-s) - 2\ve r-O(\log r)\\
&> \eta r - 2\ve r - O(\log r)
\end{align*}
establishing the second condition of Lemma \ref{lem:distanceEnumeration}. Therefore,
\begin{equation}
    K^{A,x}_{r, t}(\vert x - y\vert \mid y) \geq \eta r - \frac{3\ve}{\sigma}r - O(\log r).
\end{equation}
Hence, by definition of $\eta$,
\begin{equation}
    \dim^{A,x}(\vert x - y\vert) \geq \sigma - \frac{10\ve}{\sigma}.
\end{equation}

\end{proof}

\begin{thm}
Let $X\subseteq\R^2$ be Borel and weakly regular, i.e., $\dim_H(E) = \dim_P(E)$. Then $X$ is universal for pinned distances for the class of Borel sets $Y\subseteq \R^2$ satisfying $\dim_H(Y) < \dim_H(X)$. 
\end{thm}
\begin{proof}
Let $\sigma < \dim_H(Y) < \dim_H(X)$, and let $\ve > 0$ be sufficiently small. Let $s$ be any real such that $\dim_H(Y) < s < \dim_H(X)$. Let $\delta = \delta(\ve, s)$ be the real of Theorem \ref{thm:mainThmEffDim2}. Let $X_1\subseteq X$ and $Y_1\subseteq Y$ be compact subsets such that 
\begin{center}
    $\dim_H(X_1) > \max\{\dim_P(X) - \delta,s\}$ 
\end{center}
and $\dim_H(Y_1) > \sigma$. To simplify notation, we denote these compact subsets $X$ and $Y$, respectively. In addition, let $\mu, \nu \in \mathcal{P}(\R^2)$ whose support is $X$ and $Y$. We assume $\mu$ and $\nu$ satisfy the $s_X$-dimensional and $s_Y$-dimiensional, respectively, Frostman condition where $s_X > \max\{\dim_P(X) - \delta,s\}$ and $s_Y > \sigma$.

Let $A$ the join of the following oracles: an oracle relative to which $Y$ is effectively compact, a packing oracle for $X$,  and an oracle relative to which $\nu$ is computable. The proof of Lemma \ref{lem:C1-C3} gives the existence of some $x$ and $y$ satisfying conditions (C1)-(C3) relative to $A$. Lemma \ref{lem:C1toC3ImplyC4} proves $x$ and $y$ also satisfy condition (C4). By Theorem \ref{thm:mainThmEffDim2}
\begin{equation*}
\dim^{A,x}(\vert x-y\vert) \geq \sigma - \frac{10\ve}{\sigma}
\end{equation*}
Since $Y$ is effectively compact relative to $A$, we see that
\begin{center}
    $\dim_H(\Delta_x(Y)) \geq \sigma - \frac{10\ve}{\sigma}$.
\end{center}
Since $\ve$ can be arbitrarily small, we have
\begin{center}
    $\sup\limits_{x\in X}\dim_H(\Delta_x(Y)) \geq \sigma$.
\end{center}

\noindent Finally, since $\sigma$ can be arbitrarily close to $\dim_H(Y)$, the conclusion follows. 
\end{proof}

Combining this result with Theorem \ref{thm:priorUniversalSets} immediately gives the follows corollary. 

\begin{cor}
    Let $X\subseteq\R^2$ be Borel and weakly regular with $\dim_H(E) = \dim_P(E) > 1$. Then $X$ is weakly universal for pinned distances for the class of Borel sets $Y\subseteq \R^2$. 
\end{cor}

\subsection{AD-regular universal sets}

We begin by defining AD-regular sets and proving several pointwise results for them. Let $E \subseteq \R^n$ be a compact set, $C\geq 1$ and $\alpha \geq 0$. We say that $E$ is $(\alpha, C)$-AD regular if
\begin{equation}
    C^{-1}r^\alpha \leq \mathcal{H}^\alpha(E \cap B(x, r))\leq Cr^\alpha, \;\;\; x \in A, 0 < r < \text{diam}(A)
\end{equation}
We say that $E$ is $\alpha$-AD regular if $E$ is $(\alpha, C)$-AD regular, for some constant $C \geq 1$.

The pointwise analog of AD-regular sets is defined in the natural way. That is, a point $x\in \R^n$ is $(\alpha, C)$-AD regular with respect to an oracle $A\subseteq \N$ if
\begin{equation}
     \alpha r - C\log r \leq K_r^A(x)\leq  \alpha r + C\log r 
\end{equation}
for every $r\in \N$.

Let $E\subseteq \R^n$ be a compact set which is $\alpha$-AD regular. We say that an oracle $A \subseteq\N$ is an AD-regular Hausdorff oracle, or AD-regular oracle, for $E$ if 
\begin{equation}
     \alpha r - O(\log r) \leq K_r^A(x)\leq  \alpha r + O(\log r)  \; \forall r \in \N
\end{equation}
for $\mathcal{H}^\alpha$-a.e. $x\in E$. Note that we immediately have that any AD-regular oracle for $E$ is a Hausdorff oracle for $E$.

\begin{lem}\label{lem:ADRegularOraclesExist}
Let $E\subseteq\mathbb{R}^n$ be an $\alpha$-AD regular set. Then there is an AD-regular oracle $A$ for $E$. Moreover, for any oracle $B\subseteq\mathbb{N}$, the join $(A,B)$ of $A$ and $B$ is also an AD-regular oracle for $E$.
\end{lem}
\begin{proof}
    Let $A\subseteq \N$ be an oracle relative to which $E$ is effectively compact, which encodes the real $\alpha$, and which encodes the set of reals
    \begin{center}
        $\{\mathcal{H}^\alpha(E \cap Q) \mid Q \in \mathcal{Q}\}$,
    \end{center}
    where $\mathcal{Q}$ is the set of balls with rational centers and rational radii.

    Since the restriction $\mu :=\mathcal{H}^\alpha\vert_E$ is a computable outer measure relative to $A$, and $0 < \mu(E) < \infty$, by \cite{Stull22c}[Lemma 36], the set 
    \begin{center}
        $\{x \in E \mid (\exists^\infty r)\;K^{A}_r(x) < \alpha r - O(\log r)\}$
    \end{center}
    has $\mu$-measure zero. 
    
    Since $E$ is $(\alpha, C)$-AD regular, there is a constant $C^\prime > 1$ such that, for every $r\in \N$ there is a cover of $E$ by rational balls of radius $2^{-r}$ of cardinality at most $C^\prime 2^{\alpha r}$. Moreover, since $E$ is computably compact relative to $A$, this cover is computable relative to $A$. Therefore, for every $x \in E$ and every $r \in \N$, 
    \begin{center}
        $K^A_r(x) \leq \alpha r + O(\log r)$,
    \end{center}
    and the proof is complete.
\end{proof}

\begin{thm}\label{thm:mainThmEffDim3}
Suppose that $x, y\in\R^2$, $e = \frac{y-x}{\vert y-x\vert}$, $\sigma \in \Q \cap (0,1)$, $d_x > \sigma$, $\ve > 0$, $c\in\mathbb{N}$ and $A\subseteq\N$  satisfy the following.
\begin{itemize}
\item[\textup{(C1)}] $\dim^A(x)=d_x>1,$ $ 1\geq \dim^A(y) > \sigma$
\item[\textup{(C2)}] $K^{A,x}_r(e) > \sigma r - O(\log r)$ for all sufficiently large $r$.
\item[\textup{(C3)}] $K^{A, x}_r(y) \geq K^{A}_r(y) - O(\log  r)$ for all sufficiently large $r$. 
\item[\textup{(C4)}] $K^{A}_{t,r}(e\mid y) > \sigma t - O(\log  r)$ for all sufficiently large $r$ and $t \leq r$.
\item[\textup{(C5)}] $d_x r- c \log r\leq K_r^A(x) \leq d_x r+ c \log r$
\end{itemize}
Then
\begin{equation*}
\dim^{x,A}(\vert x-y\vert) \geq \sigma .
\end{equation*}
\end{thm}
\begin{proof}
   Note that condition (C5) immediately implies that $\Dim^A(x) - \dim^A(x) = 0$. Therefore, by Theorem \ref{thm:mainThmEffDim2}, for every $\ve > 0$,
   \begin{equation}
       \dim^{A,x}(\vert x - y\vert ) \geq \sigma - \frac{10\ve}{\sigma},
   \end{equation}
   and the conclusion follows.
\end{proof}

\begin{thm}
    Any compact AD-regular $X\subseteq\mathbb{R}^2$ of dimension more than 1 is universal for pinned distances for the class of Borel sets.
\end{thm}
\begin{proof}
Assume $X\subseteq\mathbb{R}^2$ is compact and AD regular with dimension $d_x>1$ and regularity oracle $A_X$. Let $\mu$ be the normalized probability on $X$. 

Let $Y\subseteq\mathbb{R}^2$ Borel of Hausdorff dimension at most $1$ be given, and define $Y_i$ to be a sequence of compact subsets of $Y$ satisfying $0<\mathcal{H}^{\dim_H(Y)-\frac{1}{i}}(Y_i)<\infty$. We note the probability measures on $Y_i$ (the normalized $\mathcal{H}$ measure) by $\nu_i$. Let $A_i$ be an effective compactness oracle, encoding $\nu_i$, for $Y_i$. Define $A$ to be the join of $A_X$ and $A_1, A_2,...$ and observe that $A$ remains an AD-regularity oracle for $X$ and is an effective compactness oracle for all $Y_i$. Define 
\begin{equation*}
X_1=\{x\in X: (\exists c)\, (\forall r) \text{ } d_x r - c \log r\leq K_r^A(x)\leq d_x r + c \log r\}
\end{equation*}
Note that $\mu(X_1) = 1$. 

For each $i$, let $\sigma_i = \dim_H(Y)-\frac{2}{i}$. Let $i \in \N$. Note that, if $\nu_i(\ell) \neq 0$, for any line $\ell \subseteq \R^2$, then $\dim_H(\Delta_x Y_i) = \dim_H(Y_i)$, for almost every $x \in X$. Without loss of generality, we assume that $\nu_i(\ell) = 0$ for every line $\ell \subseteq\R^2$. By Lemma 2.1 of \cite{Orponen19DimSmooth}, we have that, for any $\ve > 0$, there is a $\delta > 0$ such that $\nu_i(T) \leq \ve$ for all $\delta$-tubes $T$. 

We may therefore apply Corollary 2.18 of \cite{OrpShmWan22} to conclude that, for some $K_i$, $(\mu, \nu_i)$ has $(\sigma_i^\prime, K_i, 1- 2^{-2i})$-thin tubes, for every $\sigma_i < \sigma^\prime_i < \dim_H(Y_i)$. In particular, there is an $x \in X_1$ such that for every $i\in\N$, there is a point $y \in Y_i$ such that $x,y$ and $e := \frac{y-x}{\vert x - y\vert}$ satisfy (C1)-(C3) and (C5). Using Lemma \ref{lem:C1toC3ImplyC4}, we see that they also satisfy (C5). We may therefore apply Theorem \ref{thm:mainThmEffDim3} and conclude that
\begin{center}
    $\dim^{A,x}(\vert x - y\vert) \geq \dim_H(Y)-\frac{2}{i}$.
\end{center}
Since $Y_i$ is effectively compact relative to $A$, this implies that
\begin{equation}
    \dim_H(\Delta_x Y_i) \geq \dim_H(Y)-\frac{2}{i},
\end{equation}
and the conclusion follows.
\end{proof}

\bibliographystyle{amsplain}
\bibliography{pdss}

\end{document}